\documentclass[onefignum,onetabnum]{siamart190516}

\usepackage{amsfonts}
\usepackage{bm}
\usepackage{graphicx}
\usepackage{epstopdf}
\usepackage{algorithmic}
\usepackage[disable]{todonotes}
\usepackage{multirow}

\usepackage[strict]{changepage}

\usepackage{array}
\newcolumntype{L}[1]{>{\raggedright\let\newline\\\arraybackslash\hspace{0pt}}m{#1}}
\newcolumntype{C}[1]{>{\centering\let\newline\\\arraybackslash\hspace{0pt}}m{#1}}
\newcolumntype{R}[1]{>{\raggedleft\let\newline\\\arraybackslash\hspace{0pt}}m{#1}}

\usepackage{etoolbox}

\patchcmd{\SetTagPlusEndMark}{$}{}{}{}
\patchcmd{\SetTagPlusEndMark}{$}{}{}{}

\usepackage{enumitem}
\ifpdf
  \DeclareGraphicsExtensions{.eps,.pdf,.png,.jpg}
\else
  \DeclareGraphicsExtensions{.eps}
\fi

\crefname{enumi}{}{}
\newsiamremark{remark}{Remark}
\newsiamremark{hypothesis}{Hypothesis}
\crefname{hypothesis}{Hypothesis}{Hypotheses}
\newsiamthm{claim}{Claim}

\headers{Tightness and Equivalence of SDRs for MIMO Detection}
{R. Jiang, Y.-F. Liu, C. Bao, and B. Jiang}

\title{Tightness and Equivalence of Semidefinite Relaxations for MIMO Detection}

\author{Ruichen Jiang\thanks{Department of Electronic Engineering, Tsinghua University, Beijing 100084, China 
  (\email{rayjiang30@outlook.com}).}
\and Ya-feng Liu\thanks{State Key Laboratory of Scientific and Engineering Computing, Institute of Computational
Mathematics and Scientific/Engineering Computing, Academy of
Mathematics and Systems Science, Chinese Academy of Sciences, Beijing 
100190, China (\email{yafliu@lsec.cc.ac.cn}).}
\and Chenglong Bao\thanks{Yau Mathematical Sciences Center, Tsinghua University, Beijing 100084, China 
(\email{clbao@mail.tsinghua.edu.cn}).}
\and Bo Jiang\thanks{School of Mathematical Sciences, Key Laboratory for NSLSCS of Jiangsu Province, 
Nanjing Normal University, Nanjing 210023, China (\email{jiangbo@njnu.edu.cn}).}}

\usepackage{amsopn}
\DeclareMathOperator{\diag}{diag}

\newcommand\T{\mathsf{T}}

\newcommand\Vector[1]{\bm{#1}}

\newcommand\0{{\Vector{0}}}

\newcommand\1{{\Vector{1}}}

\newcommand\vb{{\Vector{b}}}
\newcommand\vc{{\Vector{c}}}

\newcommand\ve{{\Vector{e}}}

\newcommand\vg{{\Vector{g}}}

\newcommand\vr{{\Vector{r}}}
\newcommand\vs{{\Vector{s}}}
\newcommand\vt{{\Vector{t}}}

\newcommand\vv{{\Vector{v}}}
\newcommand\vw{{\Vector{w}}}
\newcommand\vx{{\Vector{x}}}
\newcommand\vy{{\Vector{y}}}
\newcommand\vz{{\Vector{z}}}

\newcommand\valpha{{\Vector{\alpha}}}

\newcommand\vgamma{{\Vector{\gamma}}}

\newcommand\vlambda{{\Vector{\lambda}}}
\newcommand\vmu{{\Vector{\mu}}}

\newcommand\MATRIX[1]{\mathbf{#1}}

\newcommand\mA{{\MATRIX{A}}}
\newcommand\mB{{\MATRIX{B}}}
\newcommand\mC{{\MATRIX{C}}}

\newcommand\mE{{\MATRIX{E}}}

\newcommand\mH{{\MATRIX{H}}}
\newcommand\mI{{\MATRIX{I}}}

\newcommand\mK{{\MATRIX{K}}}

\newcommand\mM{{\MATRIX{M}}}

\newcommand\mP{{\MATRIX{P}}}
\newcommand\mQ{{\MATRIX{Q}}}
\newcommand\mR{{\MATRIX{R}}}
\newcommand\mS{{\MATRIX{S}}}
\newcommand\mT{{\MATRIX{T}}}
\newcommand\mU{{\MATRIX{U}}}
\newcommand\mV{{\MATRIX{V}}}
\newcommand\mW{{\MATRIX{W}}}
\newcommand\mX{{\MATRIX{X}}}
\newcommand\mY{{\MATRIX{Y}}}
\newcommand\mZ{{\MATRIX{Z}}}

\newcommand\mhQ{{\hat{\MATRIX{Q}}}}

\newcommand\mGamma{{\MATRIX{\Gamma}}}

\newcommand\mLambda{{\MATRIX{\Lambda}}}

\newcommand\sS{{\mathbb{S}}}

\renewcommand\sS{{\mathcal{S}}}

\newcommand\blambda{{\bar{\lambda}}}

\newcommand\bvlambda{{\bar{\vlambda}}}

\newcommand\semiS{\mathbb{S}_{+}}
\newcommand\Prob{\mathbf{Prob}}
\newcommand\givenp[1][]{{\:#1\vert\:}}
\newcommand\Diag{{\mathrm{Diag}}}
\newcommand\Real{{\mathrm{Re}}}
\newcommand\Imag{{\mathrm{Im}}}
\DeclareMathOperator*{\E}{\mathbb{E}}

\newcommand\R{\mathbb{R}}
\newcommand\C{\mathbb{C}}

\newcommand\imag{{\mathbf{i}}}

\ifpdf
\hypersetup{
  pdftitle={Tightness and Equivalence of Semidefinite Relaxations for MIMO Detection},
  pdfauthor={Ruichen Jiang, Ya-feng Liu, Chenglong Bao, and Bo Jiang}
}
\fi

\begin{document}

\maketitle

\begin{abstract}
  The multiple-input multiple-output (MIMO) detection problem, a fundamental 
  problem in modern digital communications, is to detect 
  a vector of transmitted symbols from the noisy outputs of a fading MIMO channel. The 
  maximum likelihood detector can be 
  formulated as a complex least-squares problem with discrete variables, which is NP-hard in general. 
  Various semidefinite relaxation (SDR) methods have been proposed 
  in the literature to solve the problem due to their 
  polynomial-time worst-case complexity and good detection error rate performance. 
  In this paper, we consider two popular classes of SDR-based detectors 
  and study the conditions under which the SDRs are tight and the relationship between different SDR models. For the enhanced 
  complex and real SDRs proposed recently by Lu \emph{et~al.}, we refine their analysis and derive the necessary 
  and sufficient condition for the complex SDR to be tight, as well as a necessary condition 
  for the real SDR to be tight. In contrast, we also show that another SDR proposed by Mobasher \emph{et~al.} is 
  not tight with high probability under mild conditions. Moreover, we establish a general 
  theorem that shows the equivalence between two subsets of positive semidefinite matrices in 
  different dimensions by exploiting a special ``separable'' structure in the constraints. 
  Our theorem recovers two existing equivalence results of SDRs defined in different settings
  and has the potential to find other applications due to its generality.  
\end{abstract}



\section{Introduction}
Multiple-input multiple-output (MIMO) detection is a fundamental problem in modern
digital communications~\cite{multiuser_detection,fifty_years_of_MIMO}. 
The MIMO channel can be modeled as 
\begin{equation}\label{eq:mimo}
  \vr = \mH \vx^*+\vv,
\end{equation}
where $\vr\in \C^m$ is the vector of received signals, 
$\mH\in \C^{m\times n}$ is a complex channel matrix, 
$\vx^*$ is the vector of transmitted symbols, and 
$\vv$ is the vector of additive Gaussian noises. Moreover, 
each entry of $\vx^*$ is drawn from a discrete symbol set $\sS$ 
determined by the modulation scheme. 
 

The MIMO detection problem is to recover the transmitted symbol vector $\vx^*$ from the noisy channel output $\vr$, 
with the information of the symbol set $\sS$ and the channel matrix $\mH$. 
Under the assumption that each entry of $\vx^*$ is drawn uniformly and independently from the 
symbol set $\sS$, 
it is known that the maximum likelihood detector can achieve the optimal detection error rate performance. 
Mathematically, it can be formulated as a discrete least-squares problem:
\begin{equation}\label{eq:ML}
  \begin{aligned}
      \min_{\vx\in \C^n}\;\;\; &\|\mH\vx-\vr\|_2^2 \\
      \mathrm{s.t.}\,\;\;\;      & x_i\in \mathcal{S},\;i=1,2,\ldots,n,\\ 
  \end{aligned}
\end{equation}
where $x_i$ denotes the $i$-th entry of the vector $\vx$ and $\|\cdot\|_2$ denotes the Euclidean norm.  
In this paper, unless otherwise specified, we will focus on the $M$-ary phase shift keying~($M$-PSK) modulation, whose symbol set is given by 
\begin{equation}\label{eq:MPSK symbol set}
  \mathcal{S}_M:=\bigl\{z\in \C:\; |z|=1,\;\mathrm{arg}(z)\in \{2j\pi/M,\;j=0,1,\ldots,M-1\}\bigr\},
\end{equation} 
where $|z|$ and $\mathrm{arg}(z)$ denote the modulus and argument of a complex number, respectively. 
As in most practical digital communication systems, throughout the paper we require $M=2^b$ where $b\geq 1$ is an integer\footnote{Our 
results in \cref{sec:tightness}
also hold for the more general case where $M$ is a multiple of four.}. 

Many detection algorithms have been proposed to solve problem \cref{eq:ML} either exactly or approximately. 
However, for 
general $\mH$ and $\vr$, problem \cref{eq:ML} 
has been proved to be NP-hard~\cite{computational_complexity}. 
Hence, no polynomial-time algorithms can find the exact solution (unless $\text{P}=\text{NP}$). 
Sphere decoding~\cite{sphere_decoder}, a classical combinatorial algorithm based on the branch-and-bound paradigm, 
offers an efficient way to solve problem~\cref{eq:ML} exactly 
when the problem size is small, 
but its expected complexity is still exponential~\cite{complexity_of_SD}.
On the other hand, some suboptimal algorithms such as linear detectors~\cite{linear_ZF,layered_space_time} and 
decision-feedback detectors~\cite{decision_feedback,decision_feedback_cdma} enjoy low complexity but
at the expense of substantial performance loss: see~\cite{fifty_years_of_MIMO} for an excellent review. 

Over the past two decades, semidefinite relaxation (SDR) has gained increasing attention in non-convex optimization~\cite{sdr_maxcut,SDR_quadratic,phase_recovery_complex_SDP}. 
It is a celebrated technique to tackle quadratic optimization problems arising from  
various signal processing and wireless communication applications, such as beamforming design~\cite{transmit_beamforming,max_min_linear_transceiver}, 
sensor network localization~\cite{sensor_network_localization,sdp_sensor_network_localization,theory_sdp_snl}, and angular synchronization 
\cite{angular_synchronization,tightness_sdp_angular_synchronization,near_optimal_phase_sync}.
Such SDR-based approaches can usually offer superior performance in both theory and practice while 
maintaining polynomial-time worst-case complexity.

For MIMO detection problem \cref{eq:ML}, the first SDR detector~\cite{SDR_BPSK,SDR_BPSK2} was designed for the real MIMO channel and the binary symbol set $\mathcal{S}=\{+1,-1\}$.
Notably,    
it is proved that 
this detector can achieve the maximal possible diversity order~\cite{diversity_SDR}, 
meaning that it achieves an asymptotically optimal detection error rate 
when the signal-to-noise ratio (SNR) is high. 
It was later extended to the more general setting with a complex channel and an $M$-PSK symbol set in~\cite{SDR_QPSK,SDR_MPSK}, 
which we refer to as the conventional SDR or \cref{eq:CSDR}.
However, this conventional approach fails to fully utilize the structure in the symbol set $\mathcal{S}$.
To overcome this issue, researchers have developed various improved SDRs and we consider the two most popular classes  
below. The first class proposed in~\cite{near_ml_decoding} is 
based on an equivalent 
zero-one integer programming formulation of problem \cref{eq:ML}.
Four SDR models were introduced and    
two of them will be discussed in details later (see \cref{eq:Model II} and \cref{eq:Model III} further ahead). 
The second class proposed in~\cite{tightness_enhanced} further enhances \cref{eq:CSDR}    
by adding valid cuts, resulting in a complex SDR and a real SDR (see \cref{eq:CSDP2} and \cref{eq:ERSDP} later on).  

In this paper, we focus on two key problems in SDR-based MIMO detection: the tightness of SDRs and 
the relationship between different SDR models. Firstly, 
note that SDR detectors are suboptimal algorithms as they replace 
the original discrete optimization problem \cref{eq:ML} with tractable semidefinite programs (SDPs). Hence,  
after solving an SDP, we need some rounding procedure to make final symbol decisions. 
However, under some favorable conditions on $\mH$ and 
$\vv$, an SDR can be \emph{tight}, i.e., it has an optimal rank-one solution corresponding to 
the true vector of transmitted symbols. 
Such tightness conditions are of great interest since they give theoretical guarantees 
on the optimality of SDR detectors.  
While it has been well studied for the simple case~\cite{detection_mimo,optimal_condition,performance_analysis_sdp,probabilistic_analysis_SDR} where $\mH\in \R^{m\times n}$, $\vv\in \R^{m}$, and $\mathcal{S}=\{+1,-1\}$, for the more general case where $\mH\in \C^{m\times n}$, $\vv\in \C^m$, and $\mathcal{S}=\mathcal{S}_M$ ($M\geq 4$),  
tightness conditions for SDR detectors have remained unknown until very recently. The authors in~\cite{tightness_enhanced} showed 
that \cref{eq:CSDR} is not tight with probability one under some mild conditions. On the other hand,  
their proposed enhanced SDRs are tight\footnote{The definition of tightness in~\cite{tightness_enhanced} is slightly different from ours since
they also require the optimal solution of the SDR to be unique. } if the following condition is satisfied: 
\begin{equation}\label{eq:CSDP2_suff}
  \lambda_{\mathrm{min}}(\mH^\dagger \mH)\sin\left(\frac{\pi}{M}\right)
  >\|\mH^\dagger \vv\|_{\infty},
\end{equation}
where $\lambda_{\mathrm{min}}(\cdot)$ denotes the smallest eigenvalue of a matrix, 
$(\cdot)^\dagger$ denotes the conjugate transpose, and 
$\|\cdot\|_{\infty}$ denotes the $L_{\infty}$-norm. To the best of our knowledge, this 
is the best condition that guarantees a certain SDR to be tight for problem \cref{eq:ML} in the $M$-PSK settings. 

Secondly, researchers have noticed some rather unexpected equivalence between 
different SDR models independently developed in the 
literature. The earliest one of such results is reported in~\cite{QAM_equivalence}, where three different SDRs for the high-order 
quadrature amplitude modulation (QAM) symbol sets are proved to be equivalent.  
Very recently, the authors in~\cite{liu2019equivalence} showed that 
the enhanced real SDR proposed in~\cite{tightness_enhanced} is equivalent to 
one SDR model in~\cite{near_ml_decoding}. It is worth noting that while these 
two papers are of the same nature, the proof techniques are quite different and 
it is unclear how to generalize their results at present. 
 

In this paper, we make contributions to both problems. 
For the tightness of SDRs, we 
sharpen the analysis in~\cite{tightness_enhanced} to give the necessary and sufficient condition 
for the complex enhanced SDR to be tight, and a necessary condition 
for the real enhanced SDR to be tight. Specifically, for the case where $M\geq 4$, we show that the enhanced complex 
SDR \cref{eq:CSDP2} is tight if and only if 
\begin{equation}\label{eq:CSDP2_necc_and_suff_intro}
  \mH^\dagger \mH + \Diag\big(\mathrm{Re}(\Diag(\vx^*)^{-1}\mH^\dagger \vv)\big)
  -\cot\left(\frac{\pi}{M}\right)\Diag(|\mathrm{Im}(\Diag(\vx^*)^{-1}\mH^\dagger \vv)|\big) \succeq 0,
\end{equation}
while the enhanced real SDR \cref{eq:ERSDP} is tight only if 
\begin{equation}\label{eq:ERSDP_necc_intro}
  \mH^\dagger \mH + \Diag\big(\mathrm{Re}(\Diag(\vx^*)^{-1}\mH^\dagger \vv)\big)
  -\cot\left(\frac{2\pi}{M}\right)\Diag(|\mathrm{Im}(\Diag(\vx^*)^{-1}\mH^\dagger \vv)|\big) \succeq 0,
\end{equation}
where $\mA \succeq 0$ means that the matrix $\mA$ is positive semidefinite (PSD), 
$\Diag(\vx)$ denotes a diagonal matrix whose diagonals are the vector $\vx$,
and $\Real(\cdot)$, $\Imag(\cdot)$, and $|\cdot|$ denote the 
entrywise real part, imaginary part, and absolute value of a number/vector/matrix, respectively. 
Moreover, we prove that 
one of the SDR models proposed in~\cite{near_ml_decoding} is generally not tight: under some 
mild assumptions, its probability of being tight decays exponentially with respect to the number 
of transmitted symbols $n$. 

For the relationship between different SDR models, we propose a 
general theorem showing the equivalence between two subsets of PSD cones.
Specifically, we prove the correspondence between 
a subset of a high-dimensional PSD cone with a special ``separable'' structure 
and the one in a lower dimension.  
Our theorem covers both equivalence results in~\cite{QAM_equivalence} and~\cite{liu2019equivalence}
as special cases, and has the potential to find other 
applications due to its generality. 


The paper is organized as follows. 
We introduce the existing SDRs for \cref{eq:ML} in \cref{sec:review_of_SDR} and analyze their tightness in \cref{sec:tightness}. 
In \cref{sec:equiv}, 
we propose a general theorem that establishes the  
equivalence between two subsets of PSD cones in different dimensions,
and discuss how 
our theorem implies previous results. 
\Cref{sec:experiments} provides some numerical results to validate our analysis. 
Finally, \cref{sec:conclusions} concludes the paper.

We summarize some standard notations used in this paper.
We use 
$x_i$ to denote the $i$-th entry of a vector $\vx$ and 
$X_{i,j}$ to denote the $(i,j)$-th entry of 
a matrix $\mX$. 
We use $|\cdot|$, $\|\cdot\|_2$, and $\|\cdot\|_{\infty}$ to denote 
the entrywise absolute value, 
the Euclidean norm, and the $L_{\infty}$ norm
of a vector, respectively. 
For a given number/vector/matrix, we use 
$(\cdot)^\dagger$ to denote the conjugate transpose, 
$(\cdot)^{\T}$ to denote the transpose, and
$\Real(\cdot)$/$\Imag(\cdot)$ to denote the entrywise 
real/imaginary part. 
We use $\Diag(\vx)$ to denote 
the diagonal matrix whose diagonals are the vector $\vx$, 
and $\diag(\mX)$ to denote the vector whose 
entries are the diagonals of the matrix $\mX$. 
Given an $m\times n$ matrix $\mA$ and the index sets 
$\alpha \subset \{1,2,\ldots,m\}$ and $\beta \subset  \{1,2,\ldots,n\}$, 
we use $\mA[\alpha,\beta]$ to denote the submatrix with entires in the rows of $\mA$ 
indexed by $\alpha$ and the columns indexed by $\beta$. 
Moreover, we denote the principal submatrix $\mA[\alpha,\alpha]$ by 
$\mA[\alpha]$ in short. 
For two matrices $\mA$ and $\mB$ of appropriate size, 
$\langle \mA,\mB\rangle:=\Real(\mathrm{Tr}(\mA^\dagger \mB))$ denotes 
the inner product, $\mA \otimes \mB$ denotes the Kronecker product, 
and $\mA\succeq \mB$ means $\mA-\mB$ is PSD. 
For a set $\mathcal{A}$ in a vector space, 
we use $\mathrm{conv}(\mathcal{A})$ to denote its convex hull.  
For a random variable $X$ and measurable sets 
$\mathcal{B}$ and $\mathcal{C}$, $\Prob(X\in \mathcal{B})$ denotes the 
probability of the event $\{X\in \mathcal{B}\}$, 
$\Prob(X\in \mathcal{B}\givenp \mathcal{C})$ denotes 
the conditional probability given $\mathcal{C}$, 
and $\E[X]$ denotes the expectation 
of $X$.
Finally, 
the symbols $\imag$, $\1_n$, $\mI_n$, and $\semiS^n$ represent the 
imaginary unit, the $n\times1$ all-one vector,
the $n\times n$ identity matrix, and the 
$n$-dimensional PSD cone, respectively.

\section{Review of semidefinite relaxations}
\label{sec:review_of_SDR}
In this paper, we focus on the $M$\nobreakdash-PSK setting with the symbol set $\mathcal{S}_M$ given in \cref{eq:MPSK symbol set}.  
To simplify the notations, we let $\vs\in \C^M$ be the vector of all symbols, where
\begin{equation*} 
  s_j = e^{\imag\theta_j}\;\text{and}\;\theta_j=\frac{(j-1)2\pi}{M},\;
  j=1,2,\ldots,M,
\end{equation*}
and further we let $\vs_R=\Real(\vs)$ and $\vs_I=\Imag(\vs)$. 

The objective in \cref{eq:ML} can be written as 
\begin{equation*}
  \|\mH\vx-\vr\|_2^2=\vx^\dagger \mQ \vx+2\mathrm{Re}(\vc^\dagger \vx)+\vr^\dagger \vr
  = \langle \mQ, \vx\vx^\dagger \rangle+2\mathrm{Re}(\vc^\dagger \vx)+\vr^\dagger \vr,
\end{equation*}
where we define 
\begin{equation}\label{eq:def of Q and c}
  \mQ=\mH^\dagger \mH \; \text{and} \; \vc=-\mH^\dagger \vr.
\end{equation}
By introducing $\mX=\vx\vx^\dagger$ and discarding the constant $\vr^\dagger \vr$, 
we can reformulate \cref{eq:ML} as 
\begin{equation}\label{eq:reformulated ML}
  \begin{aligned}
      \min_{\vx,\mX}\;\;\; &\langle \mQ,\mX \rangle+2 \mathrm{Re}(\vc^\dagger \vx) \\
      \mathrm{s.t.}\;\;\; & X_{i,i}=1,\;i=1,2,\ldots,n,\\
      & x_i \in \mathcal{S}_M,\;i=1,2,\ldots,n,\\
      &\mX= \vx\vx^\dagger,
  \end{aligned}
\end{equation}
where the constraint $X_{i,i}=1$ comes from $X_{i,i}=|x_i|^2=1$. 
The conventional SDR (CSDR) in~\cite{SDR_QPSK,SDR_MPSK} simply drops the discrete symbol constraints $x_i \in \mathcal{S}_M$ 
and relaxes the rank-one 
constraint to $\mX\succeq \vx\vx^\dagger$, resulting in the following 
relaxation:
\begin{equation}\label{eq:CSDR}\tag{CSDR}
  \begin{aligned}
      \min_{\vx,\mX}\;\;\; &\langle \mQ,\mX \rangle+2 \mathrm{Re}(\vc^\dagger \vx) \\
      \mathrm{s.t.}\;\;\; & X_{i,i}=1,\;i=1,2,\ldots,n,\\
      &\mX\succeq \vx\vx^\dagger,
  \end{aligned}
\end{equation}
where $\vx\in \C^n$ and $\mX\in \C^{n\times n}$.
Since $\mX\succeq \vx\vx^\dagger$ is equivalent to 
\begin{equation*}
  \begin{bmatrix}
    1 & \vx^\dagger \\
    \vx & \mX
  \end{bmatrix}
  \succeq 0,
\end{equation*}
the above \cref{eq:CSDR} is an SDP on the complex domain. 
Moreover, for the simple case where $\mH\in \R^{m\times n}$, $\vv\in \R^m$, and $M=2$, a real SDR similar to \cref{eq:CSDR} has the form: 
\begin{equation}\label{eq:BSDR}
  \begin{aligned}
      \min_{\vx,\mX}\;\;\; &\langle \mQ,\mX \rangle+2 \vc^\T \vx \\
      \mathrm{s.t.}\;\;\; & X_{i,i}=1,\;i=1,2,\ldots,n,\\
      &\mX\succeq \vx\vx^{\T},
  \end{aligned}
\end{equation}
where $\vx\in\R^n$, $\mX\in \R^{n\times n}$, and we redefine $\mQ=\mH^{\T}\mH$ and $\vc=-\mH^{\T}\vr$ (cf. \cref{eq:def of Q and c}).
The problem \cref{eq:BSDR} has also been extensively studied in the literature~\cite{SDR_BPSK,SDR_BPSK2,detection_mimo,optimal_condition,performance_analysis_sdp,probabilistic_analysis_SDR}. 
It is proved in~\cite{detection_mimo,optimal_condition} that 
\cref{eq:BSDR} is tight if and only if
\begin{equation}\label{eq:real_channel_binary}
  \mH^{\T}\mH+[\Diag(\vx^*)]^{-1}\Diag(\mH^{\T}\vv)\succeq 0,
\end{equation}
while \cref{eq:CSDR} is not tight for $M\geq 4$ with probability one under some mild conditions~\cite{tightness_enhanced}. 

Recently, a class of enhanced SDRs was proposed in~\cite{tightness_enhanced}. 
Instead of simply dropping the constraints $x_i\in \sS_M$ as in \cref{eq:CSDR},
the authors replaced the discrete symbol set $\sS_M$ by its convex hull to get a 
continuous relaxation:
\begin{equation}\label{eq:CSDP2}\tag{ESDR-$\mX$}
  \begin{aligned}
      \min_{\vt,\vx,\mX}\;\;\; &\langle \mQ,\mX \rangle+2 \mathrm{Re}(\vc^\dagger \vx) \\
      \mathrm{s.t.}\,\;\;\; & X_{i,i}=1,\;i=1,2,\ldots,n,\\
      & x_i = \sum_{j=1}^M t_{i,j}s_j,\;\sum_{j=1}^{M}t_{i,j}=1,\;i=1,2,\ldots,n,\\
      & t_{i,j}\geq 0,\;j=1,2,\ldots,M,\;i=1,2,\ldots,n,\\
      &\mX\succeq \vx\vx^\dagger,
  \end{aligned}
\end{equation}
where $\vx\in\C^n$, $\mX\in\C^{n\times n}$, and 
$\vt\in \R^{Mn}$ is the concatenation of $M$-dimensional vectors $\vt_1,\vt_2,\ldots,\vt_n$ 
with $\vt_i=[t_{i,1},t_{i,2},\ldots,t_{i,M}]^{\T}$. 
The authors in~\cite{tightness_enhanced} further proved that \cref{eq:CSDP2} is tight if condition \cref{eq:CSDP2_suff} holds. 
We term the above SDP as ``ESDR\nobreakdash-$\mX$'', where ``E'' stands for ``enhanced'' and 
``$\mX$'' refers to the matrix variable. The same naming convention is adopted for 
all the SDRs below.

We can also formulate \cref{eq:reformulated ML} 
in the real domain and then use the same technique to get a real counterpart 
of \cref{eq:CSDP2}. Let
\begin{equation}\label{eq:def of hat Q and hat c}
  {\vy} = 
  \begin{bmatrix}
    \Real(\vx) \\
    \Imag(\vx)
  \end{bmatrix},\;
  \hat{\mQ}=
  \begin{bmatrix}
    \Real(\mQ)  & -\Imag(\mQ) \\
    \Imag(\mQ)  & \Real(\mQ)
  \end{bmatrix},\;\text{and}\;
  \hat{\vc} = 
  \begin{bmatrix}
    \Real(\vc) \\
    \Imag(\vc)
  \end{bmatrix},
\end{equation}
then the real enhanced SDR \cref{eq:ERSDP} is given by
\begin{equation}\label{eq:ERSDP}\tag{ESDR-$\mY$}
  \begin{aligned}
      \min_{\vt,\vy,\mY}\;\;\; &\langle \hat{\mQ},\mY \rangle+2 \hat{\vc}^{\T} \vy \\
      \mathrm{s.t.}\;\;\;\; & \mathcal{Y}(i) = \sum_{j=1}^{M}t_{i,j}\mK_j,
      \;\sum_{j=1}^{M}t_{i,j}=1, \;i=1,2,\ldots,n,\\
                          & t_{i,j}\geq 0,\;j=1,2,\ldots,M,\;i=1,2,\ldots,n,\\
                          & \mY\succeq \vy\vy^{\T},
  \end{aligned}
\end{equation}
where $\vt\in \R^{Mn},\;\vy\in \R^{2n},\;\mY\in \R^{2n\times 2n}$, and we define
\begin{equation*}
  \mathcal{Y}(i):=
  \begin{bmatrix}
    1 & y_i & y_{n+i} \\
    y_i & Y_{i,i} & Y_{i,n+i} \\
    y_{n+i} & Y_{n+i,i} & Y_{n+i,n+i}
  \end{bmatrix},\;
  i=1,2,\ldots,n.
\end{equation*}
In \cref{eq:ERSDP}, these $3\times 3$ matrices are constrained in a convex hull whose extreme points are
\begin{equation}\label{eq:def of K}
  \mK_j = 
  \begin{bmatrix}
    1 \\  s_{R,j} \\  s_{I,j}
  \end{bmatrix}
  \begin{bmatrix}
    1 \\  s_{R,j} \\  s_{I,j}
  \end{bmatrix}^{\T},\;
  j=1,2,\ldots,M,
\end{equation}
where $s_{R,j}=\Real(s_j)$ and $s_{I,j}=\Imag(s_j)$. 
It has been shown that \cref{eq:ERSDP} is tighter than \cref{eq:CSDP2}~\cite[Theorem~4.1]{tightness_enhanced}, and hence 
\cref{eq:ERSDP} is tight whenever \cref{eq:CSDP2} is tight.

Now we turn to another class of SDRs developed from a different perspective in~\cite{near_ml_decoding}, 
which is applicable to a general symbol set.  
The idea is to introduce 
binary variables to express $x_i\in \mathcal{S}_M$ by 
\begin{equation}\label{eq:express discrete by 01}
  x_i = \vt_i^{\T}\vs,\;i=1,2,\ldots,n,
\end{equation}
where $\vt_i=[t_{i,1},t_{i,2},\ldots,t_{i,M}]^{\T}$, $\sum_{j=1}^M t_{i,j}=1$, and $t_{i,j}\in \{0,1\}$. 
The above constraints~\cref{eq:express discrete by 01} can be rewritten in a compact form as 
$\vx = {\mS} \vt$,
where $\mS=\mI_n \otimes \vs^{\T}$ and we concatenate all vectors $\vt_i$ 
to get $\vt=[\vt^{\T}_1,\ldots,\vt^{\T}_n]^{\T}\in \R^{Mn}$. 
Similarly, we can also formulate \cref{eq:express discrete by 01} in the real domain as 
$\vy = \hat{\mS}\vt$,
where 
\begin{equation}\label{eq:def_of_S_hat}
  \vy = 
  \begin{bmatrix}
    \Real(\vx) \\
    \Imag(\vx)
  \end{bmatrix}\;\text{and}\;
  \hat{\mS} = 
  \begin{bmatrix}
    \Real(\mS) \\
    \Imag(\mS)
  \end{bmatrix}
  = 
  \begin{bmatrix}
    \mI_n \otimes \vs_R^{\T} \\
    \mI_n \otimes \vs_I^{\T}
  \end{bmatrix}.
\end{equation}
By introducing $\mT=\vt\vt^{\T}\in \R^{Mn\times Mn}$, the problem \cref{eq:ML} is equivalent to 
\begin{equation}\label{eq:reformulated ML T}
  \begin{aligned}
      \min_{\vt,\mT}\;\;\; &\langle \bar{\mQ},\mT \rangle+2\bar{\vc}^{\T} \vt \\
      \mathrm{s.t.}\;\;\; & \sum_{j=1}^M t_{i,j}=1,\;i=1,2,\ldots,n,\\
      & t_{i,j}\in\{0,1\},\;\;j=1,2,\ldots,M,\;i=1,2,\ldots,n,\\
      &\mT= \vt\vt^{\T},
  \end{aligned}
\end{equation}
where
\begin{equation}\label{eq:def_of_bar_Q_bar_c}
  \bar{\mQ} = \hat{\mS}^{\T}\hat{\mQ}\hat{\mS}\;\text{and}\; 
 \bar{\vc} = \hat{\mS}^{\T}\hat{\vc}.
\end{equation}
To derive an SDR for \cref{eq:reformulated ML T}, we first allow 
$t_{i,j}$ to take any value between 0 and 1. 
For the rank-one constraint $\mT=\vt\vt^{\T}$, the authors in~\cite{near_ml_decoding} 
proposed four ways of relaxation and we will introduce two of them in the following\footnote{
  Our formulations are slightly different from the original ones in~\cite{near_ml_decoding} 
  since they used the equality constraints
  to eliminate one variable for each $t_i$ before relaxing the PSD constraint. 
  However, in numerical tests we found that this variation only causes
  a negligible difference in the optimal solutions of the SDRs. 
}. 
We first partition $\mT$ as an $n\times n$ block matrix
\begin{equation*}
  \mT=  
  \begin{bmatrix}
    \mT_{1,1} & \mT_{1,2} & \ldots & \mT_{1,n} \\
    \mT_{2,1} & \mT_{2,2} & \ldots & \mT_{2,n} \\
    \vdots    &  \vdots   & \ddots &  \vdots   \\
    \mT_{n,1} & \mT_{n,2} & \ldots & \mT_{n,n}
  \end{bmatrix}, 
\end{equation*}
where $\mT_{i,j}\in \R^{M\times M}$ for $i=1,2,\ldots,n$ and $j=1,2,\ldots,n$. 
In the first model, we relax $\mT=\vt\vt^{\T}$ to $\mT\succeq \vt\vt^{\T}$ and 
impose constraints on the diagonal elements: 
\begin{equation}\label{eq:Model II}\tag{ESDR1-$\mT$}
  \begin{aligned}
  \min_{\vt,\mT}\;\;\; & \langle\bar{\mQ},\mT\rangle+2\bar{\vc}^{\T}\vt\\
  \mathrm{s.t.}\;\;\; & t_{i,j}\geq 0,\;\sum_{j=1}^{M}t_{i,j}=1,\;j=1,2,\ldots,M,\;i=1,2,\ldots,n,\\
   & \diag({\mT}_{i,i})={\vt}_{i},\;i=1,2,\ldots,n,\\
   & \mT\succeq\vt\vt^{\T},
  \end{aligned}
\end{equation}
where $\vt\in \R^{Mn}$ and $\mT\in \R^{Mn\times Mn}$. 
The second model further requires $\mT_{i,i}$ to be a diagonal matrix, leading 
to the following SDR:
\begin{equation}\label{eq:Model III}\tag{ESDR2-$\mT$}
  \begin{aligned}
  \min_{\vt,\mT}\;\;\; & \langle\bar{\mQ},\mT\rangle+2\bar{\vc}^{\T}\vt\\
  \mathrm{s.t.}\;\;\; & t_{i,j}\geq 0,\;\sum_{j=1}^{M}t_{i,j}=1,\;j=1,2,\ldots,M,\;i=1,2,\ldots,n,\\
   & {{\mT}_{i,i}=\Diag({\vt}_{i})},\;i=1,2,\ldots,n,\\
   & \mT\succeq\vt\vt^{\T}.
  \end{aligned}
\end{equation}
Since \cref{eq:Model III} puts more constraints on the variables $\vt$ and $\mT$,  
\cref{eq:Model III} is tighter than \cref{eq:Model II}. Notably, 
it is shown in~\cite{liu2019equivalence} that \cref{eq:Model III} is equivalent to \cref{eq:ERSDP}, 
and hence \cref{eq:CSDP2_suff} is also a sufficient condition for \cref{eq:Model III} to be tight.  


\cref{tab:summary of sdrs} summarizes all SDR models discussed in this paper, where we highlight our 
contributions on the tightness of different SDRs in bold. 
\begin{table}[htbp]\label{tab:summary of sdrs}
  {\footnotesize
  \caption{Summary of SDR models in this paper.}
  \begin{center}
    \begin{tabular}{|c|C{19mm}|c|C{19mm}|p{43mm}|}
      \hline
      SDR model         & Origin & Domain & Dimension of PSD cone & Comments \\ \hline
      CSDR  & Ma \textit{et al}.~\cite{SDR_MPSK}        & $\C$   & $n+1$                        &      tight with probability 0~\cite{tightness_enhanced}                          \\ \hline
      ESDR-$\mX$  & CSDP2 in Lu \textit{et al}.~\cite{tightness_enhanced}        & $\C$   & $n+1$                        &   \textbf{tight if and only if \cref{eq:CSDP2_necc_and_suff_intro} holds}                             \\ \hline
      ESDR-$\mY$  & ERSDP in Lu \textit{et al}.~\cite{tightness_enhanced}        & $\R$   & $2n+1$                       &   \textbf{tight only if \cref{eq:ERSDP_necc_intro} holds}                             \\ \hline
      ESDR1-$\mT$ & Model II in Mobasher \textit{et al}.~\cite{near_ml_decoding}         & $\R$   & $Mn+1$                       & \textbf{tight with probability no greater than~$(2/M)^n$}                               \\ \hline
      ESDR2-$\mT$ & Model III in Mobasher \textit{et al}.~\cite{near_ml_decoding}        & $\R$   & $Mn+1$                       &  equivalent to ESDR-$\mY$~\cite{liu2019equivalence}                             \\ \hline
      \end{tabular}
  \end{center}
  }
\end{table}
\section{Tightness of semidefinite relaxations}
\label{sec:tightness}

\subsection{Tightness of \texorpdfstring{{\cref{eq:CSDP2}}}{(ESDR-X)}}
Let $\mX^*=\vx^*(\vx^*)^\dagger$, and the key 
idea of showing the tightness of \cref{eq:CSDP2}
is to certify $(\vx^*,\mX^*)$ as the optimal solution by 
considering the 
Karush-Kuhn-Tucker (KKT) conditions of \cref{eq:CSDP2}. 
Our derivation is based on~\cite[Theorem~4.2]{tightness_enhanced} and 
we provide a simplified version for completeness.
\begin{theorem}[{\cite[Theorem~4.2]{tightness_enhanced}}]\label{thm:main thm for CSDP2}
  Suppose that $M\geq 4$. Then  
  $(\vx^*,\mX^*)$ is the optimal solution of \cref{eq:CSDP2} 
  if and only if there exist
  \begin{equation*}
      \lambda_i\in \R,\;\mu_{i,-1}\geq 0,\;\text{and}\;\mu_{i,1}\geq 0,\;i=1,2,\ldots,n,
  \end{equation*}
  such that $\mH$ and $\vv$ in \cref{eq:mimo} satisfy 
  \begin{equation*}
      (x_i^*)^{-1}(\mH^\dagger \vv)_i = 
      \lambda_i + \frac{\mu_{i,-1}}{2}e^{-\imag\frac{\pi}{M}} + \frac{\mu_{i,1}}{2}e^{\imag\frac{\pi}{M}},\;
      i=1,2,\ldots,n,
  \end{equation*}
  and 
  $
      \mQ + \Diag(\vlambda) \succeq 0
  $.
\end{theorem}

The authors in~\cite{tightness_enhanced} further derived the sufficient condition 
\cref{eq:CSDP2_suff}, under which the conditions in \cref{thm:main thm for CSDP2} are met by 
choosing $\lambda_i=-\lambda_{\mathrm{min}}(\mQ)$ for $i=1,2,\ldots,n$. 
To strengthen their analysis, 
we view the conditions in \cref{thm:main thm for CSDP2} as 
a semidefinite feasibility problem. To be specific, if we define 
\begin{equation}\label{eq:def of z}
    z_i = (x_i^*)^{-1}(\mH^\dagger \vv)_i,\;i=1,2,\ldots,n, 
\end{equation}
and 
\begin{equation*}
    \mathcal{C}(\lambda)=\left\{
    z\in \C :\; \exists\; \mu_{-1},\;\mu_{1}\geq 0\;\text{s.t.}\;
    z=\lambda+
    \frac{\mu_{-1}}{2}e^{-\imag\frac{\pi}{M}} + \frac{\mu_{1}}{2}e^{\imag\frac{\pi}{M}}
     \right\},
\end{equation*}
then \cref{thm:main thm for CSDP2} states that 
\cref{eq:CSDP2} is tight if and only if the following problem is feasible:
\begin{equation}\label{eq:feasibility CSDP2}
  \begin{aligned}
  \mathrm{find}\;\;\; & \vlambda\in \R^n\\
  \mathrm{s.t.}\;\;\; &\mQ + \Diag(\vlambda) \succeq 0,\\
   & z_i\in \mathcal{C}(\lambda_i),\;i=1,2,\ldots,n.
  \end{aligned}
\end{equation}
Each constraint $z_i\in \mathcal{C}(\lambda_i)$ turns 
out to be a simple inequality on $\lambda_i$. 
\begin{figure}[htbp]
  \centering
  \includegraphics[width=0.7\linewidth]{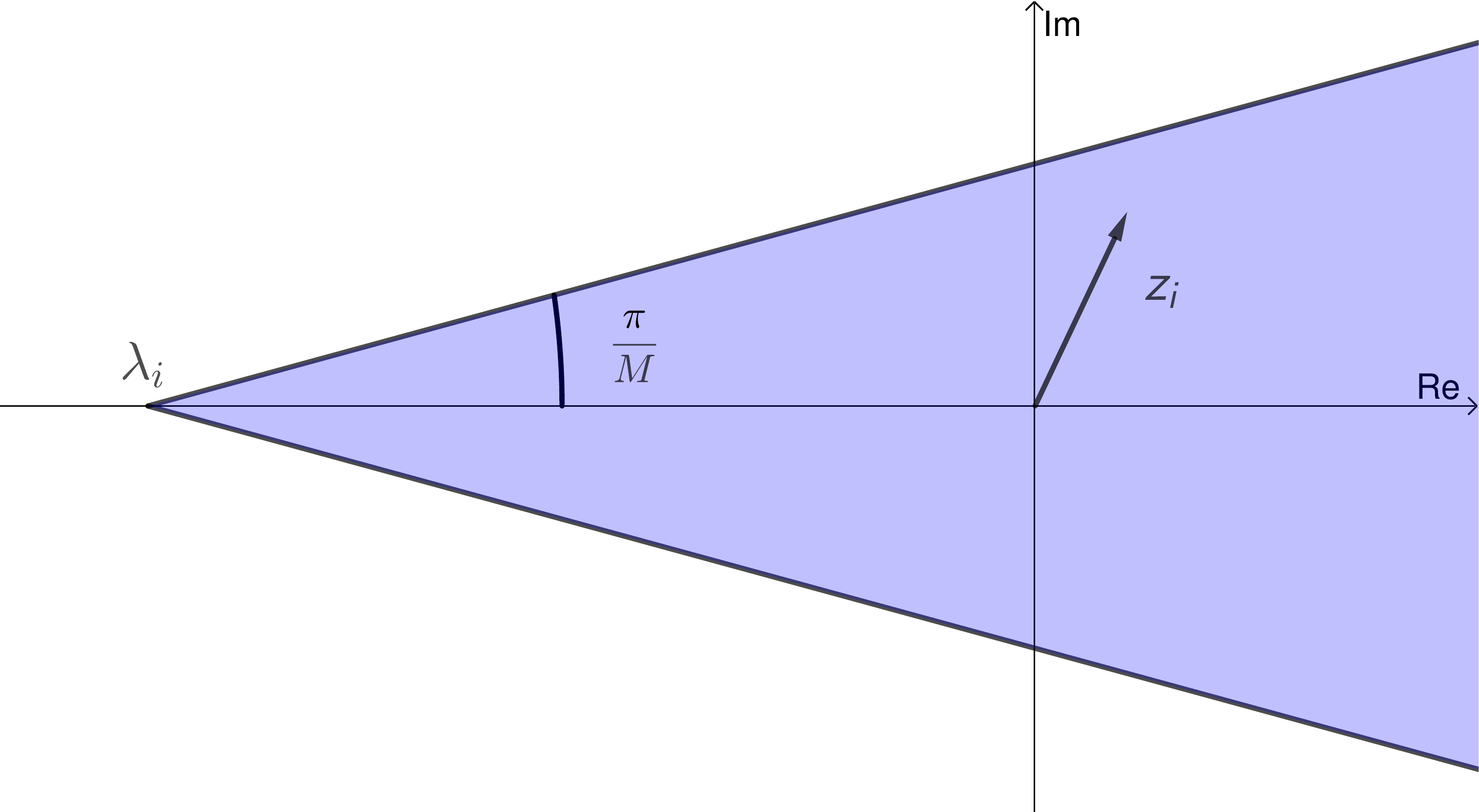}
  \caption{Illustration of $\mathcal{C}(\lambda_i)$ in the complex plane.}
  \label{fig:visual}
\end{figure}
To see this, we plot $\mathcal{C}(\lambda_i)$ as 
the shaded area in \cref{fig:visual}. 
It is clear from the figure that 
\begin{equation*}
    z_i\in \mathcal{C}(\lambda_i)
    \;\Leftrightarrow\;
    |\Imag(z_i)|\leq \big(-\lambda_i+\Real(z_i)\big)\tan\left(\frac{\pi}{M}\right),
\end{equation*}
which leads to
\begin{equation*}
  z_i\in \mathcal{C}(\lambda_i)
    \;\Leftrightarrow\;
    \lambda_i\leq \Real(z_i)-|\Imag(z_i)|\cot\left(\frac{\pi}{M}\right).
\end{equation*}
This, together with \cref{eq:feasibility CSDP2}, gives the necessary and 
sufficient condition for \cref{eq:CSDP2} to be tight and  
we formally state it in \cref{thm:CSDP2_necess_suff}.  
\begin{theorem}\label{thm:CSDP2_necess_suff}
Suppose that $M\geq 4$. Then \cref{eq:CSDP2} is tight if and only if 
\begin{equation}\label{eq:CSDP2_necc_and_suff}
  \mQ + \Diag\big(\mathrm{Re}(\vz)\big)
  -\cot\left(\frac{\pi}{M}\right)\Diag(|\mathrm{Im}(\vz)|\big) \succeq 0,
\end{equation}
where $\vz=[z_1,z_2,\ldots,z_n]^{\T}\in \C^n$.
\end{theorem}

Note that \cref{eq:CSDP2_necc_and_suff} is 
exactly the same as \cref{eq:CSDP2_necc_and_suff_intro} if we recall the 
definitions of $\mQ$ in \cref{eq:def of Q and c} and $z_i$ in \cref{eq:def of z}.
Furthermore, if we set $M=2$ and $\mH,\;\vv$ to be real in \cref{eq:CSDP2_necc_and_suff_intro}, 
it becomes the same as the previous result \cref{eq:real_channel_binary}.  
Hence, our result extends \cref{eq:real_channel_binary} to the more general 
case where $M\geq 4$ and $\mH,\;\vv$ are complex.
Finally, 
the sufficient 
condition \cref{eq:CSDP2_suff} in~\cite{tightness_enhanced} can be derived from 
our result. 
Since  
\begin{align*}
  \mathrm{Re}(z_i)-|\Imag(z_i)|\cot\left(\frac{\pi}{M}\right) &= \frac{1}{\sin\left(\frac{\pi}{M}\right)}
  \left( \mathrm{Re}(z_i)\sin\left(\frac{\pi}{M}\right)-|\Imag(z_i)|\cos\left(\frac{\pi}{M}\right)\right) \\
  & \geq -\frac{1}{\sin\left(\frac{\pi}{M}\right)} 
  |z_i| 
  \geq -\frac{1}{\sin\left(\frac{\pi}{M}\right)} \|\mH^\dagger \vv\|_{\infty},
\end{align*}
we have 
\begin{equation*}
  \Diag\big(\mathrm{Re}(\vz)\big)
  -\cot\left(\frac{\pi}{M}\right)\Diag(|\mathrm{Im}(\vz)|\big) \succeq 
  -\frac{1}{\sin\left(\frac{\pi}{M}\right)} \|\mH^\dagger \vv\|_{\infty} \mI_n.
\end{equation*}
Combining this with 
$\mQ\succeq \lambda_{\mathrm{min}}(\mQ)\mI_n$, we can see that 
\cref{eq:CSDP2_suff} is a stronger condition on 
$\mH$ and $\vv$ than \cref{eq:CSDP2_necc_and_suff_intro}.




\subsection{Tightness of \texorpdfstring{{\cref{eq:ERSDP}}}{(ESDR-Y)}}
Similar to \cref{thm:main thm for CSDP2}, we have the following 
characterization for \cref{eq:ERSDP} to be tight. 
Since the proof technique is essentially the same as that in~\cite{tightness_enhanced}, 
we put the proof in a separate technical report~\cite{companion_report}. 
\begin{theorem}\label{thm:main thm for ERSDP}
  Suppose that $M\geq 4$. Let the transmitted symbol vector $\vx^*$ be 
  \begin{equation*}
    x_i^* =s_{u_i},\;u_i\in\{1,2,\ldots,M\},\;i=1,2,\ldots,n,
  \end{equation*}
  and define
  \begin{align}
    \hat{\vv} &= 
    \begin{bmatrix}
      \mathrm{Re}(\vv) \\
      \mathrm{Im}(\vv)
    \end{bmatrix}\in \R^{2n},\;
    \hat{\mH} = 
    \begin{bmatrix}
      \mathrm{Re}(\mH) & -\mathrm{Im}(\mH) \\
      \mathrm{Im}(\mH) & \mathrm{Re}(\mH)
    \end{bmatrix}\in \R^{2n\times 2n}, \label{eq:def hat v and hat H}\\
    \vy^* &= 
    \begin{bmatrix}
      \mathrm{Re}(\vx^*) \\
      \mathrm{Im}(\vx^*)
    \end{bmatrix}\in \R^{2n},\;
    \mY^* = \vy^*(\vy^*)^\T
    \in \R^{2n\times 2n}. \nonumber
  \end{align}
  Then $(\vy^*,\mY^*)$ is the optimal solution of \cref{eq:ERSDP} 
  if and only if there exist 
  ${\vlambda} \in \R^{2n}$, ${\vmu}\in \R^{n}$, and 
  ${\vg} \in \R^{2n}$
  that satisfy   
  \begin{gather}
    \hat{\mH}^{\T} \hat{\vv} = {\vg}+({\mLambda}+{\mM})\vy^*, \label{eq:ERSDP equality}\\
    \langle {\mGamma}_i,\mK_{u_i}\rangle 
    \geq \langle {\mGamma}_i,\mK_j\rangle,\;
    j=1,2,\ldots,M,\;i=1,2,\ldots,n, \label{eq:ERSDP inequality}
  \end{gather}
  and 
  \begin{equation*}
    \hat{\mQ}+{\mLambda}+{\mM} \succeq 0, 
  \end{equation*}
  where $\mK_j$ is defined in \cref{eq:def of K}, $\hat{\mQ}$ is defined in \cref{eq:def of hat Q and hat c}, and 
  \begin{equation}\label{eq:def_of_Lambda_M_Gamma}
    {\mLambda}=\Diag({\vlambda}),\;
    {\mM}= 
  \begin{bmatrix}
    \bm{0} & \Diag({\vmu}) \\
    \Diag({\vmu}) & \bm{0}
  \end{bmatrix},\;
  {\mGamma}_i = 
  \begin{bmatrix}
    0 & {g}_i & {g}_{n+i} \\
    {g}_i & {\lambda}_i & {\mu}_i \\
    {g}_{n+i} & {\mu}_i & {\lambda}_{n+i}
  \end{bmatrix}.
  \end{equation}
\end{theorem}

Furthermore, \cref{eq:ERSDP equality,eq:ERSDP inequality} in 
\cref{thm:main thm for ERSDP} can be simplified to 
the following inequalities on $(\vlambda,\vmu)$ (see \cref{appen:simplification}):
\begin{equation}\label{eq:simplified inequality}
\begin{aligned}
  \sin^2\biggl(&\theta_{u_i}+\frac{\Delta\theta_j}{2}\biggr)\lambda_i+
  \cos^2\left(\theta_{u_i}+\frac{\Delta\theta_j}{2}\right)\lambda_{n+i}
  -
  \sin\left(2\theta_{u_i}+\Delta\theta_j\right)\mu_j \\
  &\leq \Real(z_i)-\cot\left(\frac{\Delta\theta_j}{2}\right)\Imag(z_i),\;
  j\in\{1,2,\ldots,M\}\backslash\{u_i\},\;i=1,2,\ldots,n.
\end{aligned}
\end{equation}
Here $\Delta \theta_j=\theta_j-\theta_{u_i}$, 
$\theta_{u_i}$ is the phase of the $i$-th transmitted symbol $x_i^*$, 
and $z_i$ is defined in \cref{eq:def of z}. 
Similar to \cref{eq:feasibility CSDP2}, we formulate the conditions 
in \cref{thm:main thm for ERSDP} as a semidefinite feasibility problem as follows: 
\begin{equation}\label{eq:ERSDP feasibility}
  \begin{aligned}
      \mathrm{find}\;\;\; & \vlambda\in \R^{2n}\;\text{and}\;\vmu\in \R^n \\
      \mathrm{s.t.}\;\;\; & \hat{\mQ}+\mLambda+\mM\succeq 0,\\
      & \text{\cref{eq:simplified inequality} is satisfied,}
  \end{aligned}
\end{equation}
where $\mLambda$ and $\mM$ are defined in \cref{eq:def_of_Lambda_M_Gamma}. 
However, unlike problem \cref{eq:feasibility CSDP2} where every inequality 
only involves one dual variable, problem \cref{eq:ERSDP feasibility} has 
inequalities with three variables coupled together and it is unclear how to choose the  
``optimal'' $\vlambda$ and $\vmu$. In the following, we give 
a simple necessary condition for \cref{eq:ERSDP} being tight 
based on \cref{eq:ERSDP feasibility}.
\begin{theorem}\label{thm:necess of ERSDP}
  Suppose that $M\geq 4$.
  If \cref{eq:ERSDP} is tight, then 
  \begin{equation}\label{eq:ERSDP necessary}
    \mQ + \Diag\big(\mathrm{Re}(\vz)\big)
    -\cot\left(\frac{2\pi}{M}\right)\Diag(|\mathrm{Im}(\vz)|\big) \succeq 0,
  \end{equation}
  where $\mQ=\mH^\dagger\mH$ and $\vz$ is defined in \cref{eq:def of z}. 
\end{theorem}
Before proving \cref{thm:necess of ERSDP}, we first introduce the following lemma.
\begin{lemma}\label{lem:real semidefinite to complex semidefinite}
  Suppose that $\mV$ is a PSD matrix in $\R^{2n}$ and is partitioned as  
  \begin{equation*}
    \mV = 
    \begin{bmatrix}
      \mA & \mB \\
      \mB^{\T} & \mC
    \end{bmatrix},
  \end{equation*}
  where $\mA=\mA^{\T}$, $\mC=\mC^{\T}$, and $\mA,\mB,\mC\in \R^{n\times n}$. Then 
  \begin{equation*}
    \mU = \frac{1}{2}(\mA+\mC)+\frac{\imag}{2}(\mB^{\T}-\mB)
  \end{equation*}
  is a PSD matrix in $\C^n$.
\end{lemma}
\begin{proof}
  We observe that 
  \begin{equation*}
    \mU = \frac{1}{2}\begin{bmatrix}
      \mI_n & \imag \mI_n
    \end{bmatrix} 
    \begin{bmatrix}
      \mA & \mB \\
      \mB^{\T} & \mC
    \end{bmatrix}
    \begin{bmatrix}
      \mI_n \\ -\imag \mI_n
    \end{bmatrix}.
  \end{equation*}
  The result follows immediately. 
\end{proof}
\begin{proof}[Proof of \cref{thm:necess of ERSDP}]
  If \cref{eq:ERSDP} is tight, we can find 
  ${\vlambda}\in \R^{2n}$ and ${\vmu}\in \R^n$ that satisfy the 
  constraints in \cref{eq:ERSDP feasibility}. 
  By \cref{lem:real semidefinite to complex semidefinite}, 
  the constraint $\hat{\mQ}+\mLambda+\mM\succeq 0$ implies 
  \begin{equation}\label{eq:semidefinite blambda}
    \mQ+\Diag(\bvlambda) \succeq 0, 
  \end{equation}
  where $\bvlambda\in \R^n$ is given by 
  \begin{equation*}
    \blambda_i = \frac{1}{2}(\lambda_i+\lambda_{n+i}),\;i=1,2,\ldots,n.
  \end{equation*}
  Fix $i\in\{1,2,\ldots,n\}$ and let $\hat{z}_i$ and $\hat{z}_{n+i}$ 
  denote $\Real(z_i)$ and $\Imag(z_i)$, respectively. 
  If $\hat{z}_{n+i}\geq 0$, we
  set $\Delta \theta_j=\frac{2\pi}{M}$ 
  in \cref{eq:simplified inequality} to get 
  \begin{equation*}
    \sin^2\Bigl(\theta_{u_i}+\frac{\pi}{M}\Bigr)\lambda_i+
    \cos^2\Bigl(\theta_{u_i}+\frac{\pi}{M}\Bigr)\lambda_{n+i}-
    \sin\Bigl(2\theta_{u_i}+\frac{2\pi}{M}\Bigr)\mu_j
    \leq \hat{z}_i-\cot\Bigl(\frac{\pi}{M}\Bigr)|\hat{z}_{n+i}|.
  \end{equation*}
  Since $M\geq 4$, we can also set 
  $\Delta \theta_j=\frac{2\pi}{M}+\pi$ to get 
  \begin{equation*}
    \cos^2\Bigl(\theta_{u_i}+\frac{\pi}{M}\Bigr)\lambda_i+
    \sin^2\Bigl(\theta_{u_i}+\frac{\pi}{M}\Bigr)\lambda_{n+i}+
    \sin\Bigl(2\theta_{u_i}+\frac{2\pi}{M}\Bigr)\mu_j
    \leq \hat{z}_i+\tan\Bigl(\frac{\pi}{M}\Bigr)|\hat{z}_{n+i}|.
  \end{equation*}
  Adding the above two inequalities and dividing both sides by two, 
  we have 
  \begin{equation}\label{eq:inequality on blambda}
    \blambda_i=\frac{1}{2}(\lambda_i+\lambda_{n+i})\leq \hat{z}_i-\cot\left(\frac{2\pi}{M}\right)|\hat{z}_{n+i}|.
  \end{equation}
  If $\hat{z}_{n+i}<0$, we can also arrive at \cref{eq:inequality on blambda} 
  by setting 
  $\Delta\theta_j$ to be $-\frac{2\pi}{M}$ and $-\frac{2\pi}{M}+\pi$, respectively. 
  Finally, \cref{thm:necess of ERSDP} follows from \cref{eq:semidefinite blambda} 
  and \cref{eq:inequality on blambda}.
\end{proof}

Note that \cref{eq:ERSDP necessary} is the same as \cref{eq:ERSDP_necc_intro} if we 
recall the definitions of $\mQ$ in \cref{eq:def of Q and c} and $z_i$ in \cref{eq:def of z}. 
Moreover, since \cref{eq:ERSDP} is tighter than \cref{eq:CSDP2}, 
\cref{eq:ERSDP} will also be tight if \cref{eq:CSDP2_necc_and_suff_intro} holds. Therefore, we 
have both a necessary condition \cref{eq:ERSDP_necc_intro} 
and a sufficient condition \cref{eq:CSDP2_necc_and_suff_intro} for \cref{eq:ERSDP} to be tight. 
\subsection{Tightness of \texorpdfstring{{\cref{eq:Model II}}}{(ESDR1-T)}}
In the same spirit, 
we first give a necessary and sufficient condition for \cref{eq:Model II} to be tight. 
Since the technique is essentially the same as that used in \cref{thm:main thm for ERSDP}, we omit the 
proof details due to the space limitation. 
\begin{theorem} \label{thm:main thm for Model II}
  Suppose that $M\geq 4$. Let the transmitted symbol vector $\vx^*$ be 
  \begin{equation*}
    x_i^* =s_{u_i},\;u_i\in\{1,2,\ldots,M\},\;i=1,2,\ldots,n,
  \end{equation*}
  and define 
\begin{gather*}
  \vt^*_{i,u_i} = 1,\; \vt^*_{i,j} = 0,\;j\neq u_i,\;i=1,2,\ldots,n, \\
  \mT^* = \vt^* (\vt^*)^{\T}. \label{eq: rank 1}
\end{gather*}
Then $(\vt^{*},\mT^{*})$ is the optimal solution of \eqref{eq:Model II}
if and only if there exist ${\valpha}\in \mathbb{R}^n$ and ${{\vgamma}}\in\mathbb{R}^{Mn}$ such that 
\begin{equation}\label{eq:Model II equality}
  \Diag(1-2\vt^*){\vgamma}=-2{\hat{\mS}}^{\T}\hat{\mH}^{\T}\hat{\vv}+{\valpha}\otimes \1_M,
\end{equation}
and 
\begin{equation}\label{eq:Model II semidefiniteness}
  \bar{\mQ}+\Diag({{{\vgamma}}})\succeq 0.
\end{equation}
where  $\hat{\mS}$ is defined in 
\cref{eq:def_of_S_hat},  
$\hat{\mH},\;\hat{\vv}$ are defined in \cref{eq:def hat v and hat H}, 
and $\bar{\mQ}$ is defined in \cref{eq:def_of_bar_Q_bar_c}. 

\end{theorem}

Now we provide a corollary that will serve as our basis for further derivation.  
\begin{corollary} \label{coro:Model II necessary condition}
  If \eqref{eq:Model II} is tight, then there exist 
  \begin{equation*}
    {\valpha}\in \R^n\;\text{and}\;{\vgamma}_1,{\vgamma}_2,\ldots,{\vgamma}_n\in \R^{M}
  \end{equation*}
  that satisfy
  \begin{equation}\label{eq:Model II equality explicit}
    {\gamma}_{i,j} = 
    \begin{cases}
      -2 \mathrm{Re}[s_j^\dagger (\mH^\dagger \vv)_i]+{\alpha}_i,\;\;\; &\text{if}\;j\neq u_i, \\
      2 \mathrm{Re}[s_j^\dagger (\mH^\dagger \vv)_i]-{\alpha}_i,\;\;\; &\text{if}\;j=u_i,
    \end{cases}
    \;j=1,2,\ldots,M,\;i=1,2,\ldots,n,
  \end{equation}
and 
\begin{equation*}
  \vw^{\T}\Diag({\vgamma_i})\vw \geq 0,\;\;i=1,2,\ldots,n,
\end{equation*} 
for any $\vw\in \R^M$ such that 
$
  \vw^{\T}\vs_{R}=\vw^{\T}\vs_{I}=0
$.
  
\end{corollary}
\begin{proof}
  By \cref{thm:main thm for Model II}, if \cref{eq:Model II} is tight, we 
  can find ${\valpha}\in \R^n$ and ${\vgamma}\in \R^{Mn}$ that satisfy 
  \eqref{eq:Model II equality} and \eqref{eq:Model II semidefiniteness}. Let 
  $\vgamma$ be partitioned as  
  ${\vgamma}=[{\vgamma}_1^{\T},{\vgamma}_2^{\T},\ldots,{\vgamma}_n^{\T}]^{\T}$ 
  where ${\vgamma}_j\in \R^M$ 
  is the $j$-th block of ${\vgamma}$. 
  It is straightforward to verify that \eqref{eq:Model II equality} is equivalent to 
  \eqref{eq:Model II equality explicit}. 

  Moreover, for any $i\in \{1,2,\ldots,n\}$ and 
  any $\vw\in \R^M$ that satisfies $\vw^{\T}\vs_{R}=\vw^{\T}\vs_{I}=0$, we set 
  $\bar{\vw}=[\bar{\vw}_1^{\T},\bar{\vw}_2^{\T},\ldots,\bar{\vw}_n^{\T}]^{\T}\in \R^{Mn}$ to be 
  \begin{equation*}
    \bar{\vw}_j = 
    \begin{cases}
      \0, & \text{if}\; j\neq i,\\
      \vw, & \text{if}\; j=i.
    \end{cases}
  \end{equation*}
It is simple to check that $\hat{{\mS}}\bar{\vw}=\0$. Therefore, recalling that 
$\bar{\mQ}=\hat{{\mS}}^{\T}\mhQ\hat{{\mS}}$, by \eqref{eq:Model II semidefiniteness} 
we have 
\begin{equation*}
  \bar{\vw}^{\T} (\bar{\mQ}+\Diag({{{\vgamma}}}))\bar{\vw} 
  =\bar{\vw}^{\T}\Diag({{{\vgamma}}})\bar{\vw}
  =\vw^{\T}\Diag({\vgamma_i})\vw \geq 0.
\end{equation*}
The proof is complete.
\end{proof}

In practice,
the symbol set $\sS$, such as the one in \cref{eq:MPSK symbol set} considered in this paper, 
is symmetric with respect to the origin. 
Therefore, we can find 
$u_i'\in\{1,2,\ldots,M\}$ that satisfies $s_{u_i'}=-s_{u_i}$. Now let $\vw\in \R^M$ be 
\begin{equation*}
  w_j = 
  \begin{cases}
    0, & \text{if}\; j\notin \{u_i,u_i'\},\\
    1, & \text{if}\; j\in \{u_i,u_i'\},
  \end{cases}\; j=1,2,\ldots,M.
\end{equation*}
We have $\vw^{\T}\vs_R=s_{R,u_i}+s_{R,u_i'}=0$ and $\vw^{\T}\vs_I=s_{I,u_i}+s_{I,u_i'}=0$. 
Hence, when \cref{eq:Model II} is tight, \cref{coro:Model II necessary condition} implies that 
\begin{equation*}
  \vw^{\T}\Diag({\vgamma_i})\vw = {\gamma}_{u_i}+{\gamma}_{u_i'}=
  4\mathrm{Re}[s_{u_i}^\dagger (\mH^\dagger \vv)_i] =
  4\mathrm{Re}[(x_i^*)^\dagger (\mH^\dagger \vv)_i] 
  \geq 0.
\end{equation*}
This immediately leads to the following upper bound on the tightness probability 
of \cref{eq:Model II}.
\begin{corollary}\label{coro:general bound}
  Suppose that 
  the symbol set $\sS$ is symmetric with respect to the origin and $0\notin \sS$. 
  We further assume that 
  \begin{enumerate}[label=(\alph*)]
    \item The entries of $\vx^*$ are drawn from $\sS$ uniformly and independently;
    \item $\vx^*$, $\mH$, and $\vv$ are mutually independent; and
    \item the distribution of $\mH$ and $\vv$ are continuous.
  \end{enumerate} 
  Then we have 
  \begin{equation*}
    \Prob\bigl(\text{\cref{eq:Model II} is tight}\bigr) \leq \left(\frac{1}{2}\right)^n.
  \end{equation*}
\end{corollary}
\begin{proof}
  Let $z_i=(x_i^*)^\dagger(\mH^\dagger \vv)_i$, $i=1,2,\ldots,n$.  
  Since $\mH$ and $\vv$ are independent continuous random variables, 
  the event 
  \begin{equation*}
    \bigcup_{i=1}^n \bigcup_{s\in \sS} \{\mathrm{Re}(s^\dagger (\mH^\dagger \vv)_i)=0\}
  \end{equation*}
  happens with probability zero. 
  Hence, because of the symmetry of $\sS$,  
  with probability one exactly half of the symbols $s\in \sS$ satisfy 
 $
    \mathrm{Re}(s^\dagger (\mH^\dagger \vv)_i)>0
  $
  for each $i\in\{1,2,\ldots,n\}$ when $\mH$ and $\vv$ are given. By the assumption that $x^*_i$ is uniformly distributed 
  over $\sS$, we obtain 
  \begin{equation*}
    \Prob\bigl(\mathrm{Re}(z_i)\geq 0 \givenp \mH,\vv \bigr) = \frac{1}{2}\;\text{almost surely}.
  \end{equation*}
  Moreover, $\{z_i\}_{i=1}^n$ are mutually independent conditioned on $\mH$ and $\vv$. 
  This leads to 
  \begin{align*}
    \Prob\bigl(\mathrm{Re}(z_i)\geq 0,\;i=1,2,\ldots,n \bigr)&=\E_{\mH,\vv}\bigl[\Prob\bigl(\mathrm{Re}(z_i)\geq 0,\;i=1,2,\ldots,n \givenp \mH,\vv\bigr)\bigr] \\
    & = \E_{\mH,\vv}\bigl[\prod_{i=1}^n 
    \Prob\bigl(\mathrm{Re}(z_i)\geq 0 \givenp \mH,\vv\bigr)\bigr] \\
    & = \left(\frac{1}{2}\right)^n.
  \end{align*}
  Finally, \cref{coro:general bound} follows from the fact that the tightness of \cref{eq:Model II} implies $\Real(z_i)\geq 0$, $i=1,2,\ldots,n$. 
\end{proof}

It is worth noting that all the assumptions in \cref{coro:general bound} are mild: they are satisfied if 
we use the $M$-PSK or QAM modulation scheme and the entries of $\mH$ and $\vv$ follow the complex Gaussian distribution. 

Intuitively, we will expect that \cref{eq:Model II} is less likely 
to recover the transmitted symbols with an increasing symbol set size $M$. 
In the following, we present a more refined upper bound on the tightness 
probability specific to the $M$-PSK setting and the proof can be found in~\cite{companion_report}. 
\begin{theorem}\label{thm:MPSK bound}
  Suppose that $M$-PSK is used with $M\geq 4$ 
  and the same assumptions in 
  \cref{coro:general bound} hold. 
  Then we have
  \begin{equation*}
    \Prob\bigl(\text{\cref{eq:Model II} is tight}\bigr) \leq \left(\frac{2}{M}\right)^n.
  \end{equation*}
\end{theorem}

From \cref{coro:general bound,thm:MPSK bound}, we can see that the tightness probability 
of \cref{eq:Model II} is bounded away from one regardless of the noise level, and 
it tends to zero exponentially fast when the number of transmitted symbols $n$ increases. 
This is in sharp contrast to \cref{eq:CSDP2,eq:ERSDP}, whose tightness probabilities will 
approach one if the noise level is sufficiently small and the number of received signals $m$ is sufficiently large compared to $n$
\cite[Theorem~4.5]{tightness_enhanced}. 

\section{Equivalence between different SDRs}
\label{sec:equiv}
In this section, we focus on the relationship between different SDR models of \cref{eq:ML}.
Related to the SDRs discussed so far, a recent paper~\cite{liu2019equivalence} proved 
that \cref{eq:Model III} is equivalent to \cref{eq:ERSDP} for the MIMO detection problem with 
a general symbol set. An earlier paper~\cite{QAM_equivalence} compared 
three different SDRs in the QAM setting and showed their equivalence. 
Compared with those in \cref{sec:review_of_SDR}, 
the SDRs considered in~\cite{QAM_equivalence} differ greatly in their motivations and 
structures, and the two equivalence results are 
proved using different techniques. 
In this section, we provide a more general equivalence theorem from which 
both results follow as special cases. 
This not only reveals the underlying connection between these two works, but also may potentially lead to 
new equivalence between SDRs. 

\subsection{Review of previous results}
In~\cite{liu2019equivalence}, the authors established the following correspondence 
between a pair of feasible points of \cref{eq:Model III} and \cref{eq:ERSDP}:
\begin{equation}\label{eq:modelIII-ERSDP equiv}
  \mY = \hat{\mS}\mT\hat{\mS}^{\T}\;\text{and}\;\vy=\hat{\mS}\vt, 
\end{equation}
where $\hat{\mS}\in \R^{2n\times Mn}$ is defined in \cref{eq:def_of_S_hat}. 
In~\cite{QAM_equivalence}, the authors considered the feasible set  
of a virtually-antipodal SDR (VA-SDR): 
\begin{equation}\label{eq:VA-SDR}\tag{VA-SDR}
  \begin{aligned}
    & \begin{bmatrix}
      1 & \vb^{\T} \\
      \vb & \mB
    \end{bmatrix}\in \semiS^{qn+1} \\
    \mathrm{s.t.}\;\;\; & 
    B_{i,i}=1,\;i=1,2,\ldots,qn,
  \end{aligned}
\end{equation}
and that of a bounded-constrained SDR (BC-SDR): 
\begin{equation}\label{eq:BC-SDR}\tag{BC-SDR}
  \begin{aligned}
    & \begin{bmatrix}
      1 & \vx^{\T} \\
      \vx & \mX
    \end{bmatrix}\in \semiS^{n+1} \\
    \mathrm{s.t.}\;\;\; & 
    1\leq X_{i,i}\leq (2^q-1)^2,\;i=1,2,\ldots,n,
  \end{aligned}
\end{equation}
where $q\geq 1$ is an integer. We refer interested readers to~\cite{QAM_equivalence} and 
references therein for their derivations.
The authors proved the equivalence between \cref{eq:VA-SDR} and \cref{eq:BC-SDR} by showing the following correspondence: 
\begin{equation}\label{eq:VA-BC equiv}
  \mX = {\mW}\mB{\mW}^{\T}\;\text{and}\;\vx={\mW}\vb, 
\end{equation} 
where 
\begin{equation*}
  \mW = 
  \begin{bmatrix}
    \mI_n & 2\mI_n & 4\mI_n & \ldots & 2^{q-1}\mI_n
  \end{bmatrix}
  \in \R^{n\times qn}.
\end{equation*}

Note that both equivalence results in \eqref{eq:modelIII-ERSDP equiv} 
and \eqref{eq:VA-BC equiv} fall into the following form:
\begin{equation*}
  \left\{ 
    \begin{bmatrix}
      1 & \vy^{\T} \\
      \vy & \mY
    \end{bmatrix}\in \mathcal{F}_1
  \right\}
  =
  \left\{
    \begin{bmatrix}
      1 & \vy^{\T} \\
      \vy & \mY
    \end{bmatrix}
    = 
    \begin{bmatrix}
      1 & \0 \\
      \0 & \mP
    \end{bmatrix}
    \begin{bmatrix}
      1 & \vt^{\T} \\
      \vt & \mT
    \end{bmatrix}
    \begin{bmatrix}
      1 & \0 \\
      \0 & \mP^{\T}
    \end{bmatrix}
    :\;
    \begin{bmatrix}
      1 & \vt^{\T} \\
      \vt & \mT
    \end{bmatrix}
    \in \mathcal{F}_2
  \right\},  
\end{equation*}
where $\mathcal{F}_1$ is a subset of $\semiS^{k+1}$, 
$\mathcal{F}_2$ is a subset of $\semiS^{d+1}$, and we call $\mP\in \R^{k\times d}$ 
as the transformation matrix. Moreover, both the 
transformation matrices $\hat{\mS}$ in \eqref{eq:modelIII-ERSDP equiv} 
and $\mW$ in \eqref{eq:VA-BC equiv} have a special ``separable'' property that 
we now define for ease of presentation. 
\begin{definition}
  A matrix $\mP\in\R^{k\times d}$ is called separable if there 
  exist a partition of rows $\alpha_1,\alpha_2,\ldots,\alpha_l$ and 
  a partition of columns $\beta_1,\beta_2,\ldots,\beta_l$ such that 
  \begin{equation*}
    \mP[\alpha_i,\beta_j]=\mathbf{0},\qquad \forall\;i\neq j.
  \end{equation*}
\end{definition}
In other words, a matrix is separable if, after possibly rearranging rows and 
columns, it has a block diagonal structure. In particular, for 
the transformation matrix $\hat{\mS}$ in \cref{eq:modelIII-ERSDP equiv}, 
the corresponding row and column partitions are given by 
\begin{equation}\label{eq:parititon in ModelIII}
  \alpha_i = \{i,n+i\},\;\beta_i=\{(i-1)M+1,(i-1)M+2,\ldots,iM\},\;i=1,2,\ldots,n;
\end{equation}
for the transformation matrix $\mW$ in \cref{eq:VA-BC equiv}, 
they are given by 
\begin{equation}\label{eq:parititon in VA}
  \alpha_i = \{i\},\;\beta_i=\{i,i+n,i+2n,\ldots,i+(q-1)n\},\;i=1,2,\ldots,n.
\end{equation}
\subsection{A general equivalence theorem} Now we are ready to present our main equivalence
result. 
\begin{theorem}\label{thm:main equiv thm}
  Suppose that the matrix $\mP\in \R^{k\times d}$ is separable with 
  row partition $\alpha_1,\alpha_2,\ldots, \alpha_l$ 
  and column partition $\beta_1,\beta_2,\ldots, \beta_l$.
  Moreover, define 
  \begin{equation*}
    k_i:=|\alpha_i|,\;d_i:=|\beta_i|,\;\text{and}\;\mP_i:=\mP[\alpha_i,\beta_i]\in \R^{k_i\times d_i},\;
    i=1,2,\ldots,l,
  \end{equation*}
  where we use $|\cdot|$ to denote the cardinality of a set. 
  Then given arbitrary constraint sets ${\mathcal{A}_i\subset \R^{d_i\times d_i}}$ 
  for $i=1,2,\ldots,l$, 
  the following set 
  \begin{equation}\label{eq:original} 
    \begin{aligned}
      & \begin{bmatrix}
        1 & \vy^{\T} \\
        \vy & \mY
      \end{bmatrix}\in \semiS^{k+1} \\
      \mathrm{s.t.}\;\;\; & 
      \mY = \mP \mT \mP^{\T},\;\vy=\mP\vt,\\
      & \begin{bmatrix}
        1 & \vt^{\T} \\
        \vt & \mT
      \end{bmatrix}\in \semiS^{d+1},\\
      & \begin{bmatrix}
          1 & \vt[\beta_i]^{\T} \\
          \vt[\beta_i] & \mT[\beta_i]
      \end{bmatrix}\in \mathcal{A}_i,\;i=1,2,\ldots,l,
    \end{aligned}
  \end{equation}
  where the variables are $\vy\in \R^k$, $\mY\in \R^{k\times k}$, $\vt\in \R^d$, and $\mT\in \R^{d\times d}$, 
  is the same as 
  \begin{equation}\label{eq:decomposed} 
    \begin{aligned}
      & \begin{bmatrix}
        1 & \vy^{\T} \\
        \vy & \mY
      \end{bmatrix}\in \semiS^{k+1} \\
      \mathrm{s.t.}\;\;\; & 
      \mY[\alpha_i] = \mP_i \mT^{(i)} \mP_i^{\T},\;\vy[\alpha_i]=\mP_i\vt^{(i)},\\
      & \begin{bmatrix}
        1 & (\vt^{(i)})^{\T} \\
        \vt^{(i)} & \mT^{(i)}
      \end{bmatrix}\in \semiS^{d_i+1},\\
      & \begin{bmatrix}
        1 & (\vt^{(i)})^{\T} \\
        \vt^{(i)} & \mT^{(i)}
      \end{bmatrix} \in \mathcal{A}_i,\;i=1,2,\ldots,l,
    \end{aligned}
  \end{equation}
  where the variables are $\vy\in \R^k$, $\mY\in \R^{k\times k}$, $\vt^{(i)}\in \R^{d_i}$, and $\mT^{(i)}\in \R^{d_i\times d_i}$ 
  with $i=1,2,\ldots,l$. 
\end{theorem}

The following lemma will be 
useful in our proof. 
\begin{lemma}[{\cite[Theorem~7.3.11]{matrix_analysis}}]\label{lem: two factorizations}
  Let $\mA\in \R^{p\times n}$ and $\mB\in \R^{q\times n}$ where $p\leq q$. Then 
  $\mA^{\T}\mA=\mB^{\T}\mB$ if and only if there exists a matrix $\mU\in \R^{q\times p}$ with 
  $\mU^{\T}\mU=\mI_p$ such that $\mB=\mU\mA$. 
\end{lemma}
\begin{proof}[Proof of \cref{thm:main equiv thm}]
Without loss of generality, we assume the transformation matrix $\mP\in \R^{k\times d}$ 
is in the form 
\begin{equation*}
  \mP = 
  \begin{bmatrix}
    \mP_1 & & & \\
          & \mP_2 & & \\
          & & \ddots & \\
          & & & \mP_l
  \end{bmatrix},
\end{equation*}
where $\mP_i\in \R^{k_i\times d_i}$, $\sum_{i=1}^l k_i=k$, and $\sum_{i=1}^l d_i=d$. 

For one direction, suppose that $(\vy, \mY, \vt,\mT)$ satisfies the constraints 
in \cref{eq:original}. Then it is straightforward to see that $(\vy,\mY)$ also 
satisfies the constraints in \cref{eq:decomposed} together with 
\begin{equation*}
  \vt^{(i)} = \vt[\beta_i],\;\mT^{(i)}=\mT[\beta_i],\;i=1,2,\ldots,l.
\end{equation*}

The other direction of the proof is more involved. Given $(\vy,\mY)$ 
and the variables $\{\vt^{(i)},\mT^{(i)}\}_{i=1}^l$ in \cref{eq:decomposed}, our goal is to construct 
$(\vt,\mT)$ satisfying the conditions in \cref{eq:original}. To simplify the notations, we define 
\begin{equation}\label{eq:augmented matrices}
  \tilde{\mY} := 
  \begin{bmatrix}
    1 & \vy^{\T} \\
    \vy & \mY
  \end{bmatrix}
  ,\;
  \tilde{\mT}^{(i)} := 
  \begin{bmatrix}
    1 & (\vt^{(i)})^{\T} \\
    \vt^{(i)} & \mT^{(i)}
  \end{bmatrix}
  ,\;
  \text{and}\;
  \tilde{\mP}_i := 
  \begin{bmatrix}
    1 & \0 \\
    \0 & \mP_i
  \end{bmatrix}.
\end{equation}
Let $r=\max\{k,d\}$. 
Since $\tilde{\mY}\in \semiS^{k+1}$, it can be factorized as 
\begin{equation}\label{eq:Y factorization}
  \tilde{\mY} = \tilde{\mV}^{\T}\tilde{\mV},
\end{equation}
where $\tilde{\mV}\in \R^{(r+1)\times (k+1)}$. The above factorization can be done because $r\geq k$.
Further, we partition $\tilde{\mV}$ as 
\begin{equation*}
  \tilde{\mV} = 
  \begin{bmatrix}
    \vv & \mV_1 & \mV_2 & \ldots  & \mV_l
  \end{bmatrix},
\end{equation*}
where $\vv\in \R^{r+1}$ and $\mV_i\in \R^{(r+1)\times k_i}$ contains the columns of $\tilde{\mV}$ indexed by $\alpha_i$ 
for $i=1,2,\ldots,l$.   
Moreover, we have $\vv^{\T}\vv=\tilde{\mY}_{1,1}=1$. 
Similarly, $\tilde{\mT}^{(i)}$ can be factorized as 
\begin{equation}\label{eq:T factorization}
  \tilde{\mT}^{(i)} = (\tilde{\mZ}^{(i)})^{\T}\tilde{\mZ}^{(i)},\;i=1,2,\ldots,l,
\end{equation}
where $\tilde{\mZ}^{(i)}\in \R^{(d_i+1)\times (d_i+1)}$ and is partitioned as 
\begin{equation}\label{eq:paritioned Z}
  \tilde{\mZ}^{(i)} = 
  \begin{bmatrix}
    \vz^{(i)} & \mZ^{(i)}
  \end{bmatrix}.
\end{equation} 
Combining \cref{eq:T factorization} with 
the equality constraints in \cref{eq:decomposed}, we get 
\begin{equation*}
  \begin{bmatrix}
    1 & \vy[\alpha_i]^{\T} \\
    \vy[\alpha_i] & \mY[\alpha_i]
  \end{bmatrix}
  =\tilde{\mP}_i \tilde{\mT}^{(i)}\tilde{\mP}^{\T}_i 
  = \bigl(\tilde{\mZ}^{(i)}\tilde{\mP}^{\T}_i\bigr)^{\T}
  \bigl(\tilde{\mZ}^{(i)}\tilde{\mP}^{\T}_i\bigr),\;i=1,2,\ldots,l,
\end{equation*}
where $\tilde{\mZ}^{(i)}\tilde{\mP}^{\T}_i\in \R^{(d_i+1)\times (k_i+1) }$. 
On the other hand, the factorization in \cref{eq:Y factorization} implies 
\begin{equation*}
  \begin{bmatrix}
    1 & \vy[\alpha_i]^{\T} \\
    \vy[\alpha_i] & \mY[\alpha_i]
  \end{bmatrix}
  = 
  \begin{bmatrix}
    \vv & \mV_i
  \end{bmatrix}^{\T}
  \begin{bmatrix}
    \vv & \mV_i
  \end{bmatrix},
\end{equation*}
where $  [
  \vv \quad \mV_i
]\in \R^{(r+1)\times (k_i+1)}$. By  
\cref{lem: two factorizations}, we can find $\mU_i\in \R^{(r+1)\times (d_i+1)}$ 
with $\mU_i^{\T}\mU_i=\mI_{d_i+1}$ such that 
\begin{equation}\label{eq: equality of factors}
  \begin{bmatrix}
    \vv & \mV_i
  \end{bmatrix} 
  =\mU_i \tilde{\mZ}^{(i)}\tilde{\mP}^{\T}_i.
\end{equation}
Substituting \cref{eq:augmented matrices} and \cref{eq:paritioned Z} into \cref{eq: equality of factors}, 
we get 
\begin{equation}\label{eq: equality between V and Z}
  \vv=\mU_i\vz^{(i)}\;\text{and}\;
  \mV_i=\mU_i{\mZ}^{(i)}\mP_i^\T.
\end{equation}
Finally, we define 
\begin{equation*}
  \mR = \begin{bmatrix}
    \vv & \mU_1{\mZ}^{(1)}  & \mU_2{\mZ}^{(2)} & \ldots & \mU_l{\mZ}^{(l)}
  \end{bmatrix}\in \R^{(r+1)\times (d+1)},
\end{equation*}
whose columns indexed by $\beta_i$ are given by $\mU_i{\mZ}^{(i)}$, 
and construct $(\vt,\mT)$ by 
\begin{equation} \label{eq:constructed t and T}
\begin{bmatrix}
  1 & \vt^{\T} \\
  \vt & \mT
\end{bmatrix}
=
\mR^{\T}\mR.
\end{equation}

Next we verify that $(\vt,\mT)$ in \cref{eq:constructed t and T} indeed satisfies 
all the constraints in \cref{eq:original}. 
The positive semidefiniteness is evident by our construction. 
For the equality constraints, by using \cref{eq: equality between V and Z} we have
\begin{align*}
  \mR
  \begin{bmatrix}
    1 & \0 \\
    \0 & \mP^{\T}
  \end{bmatrix}
  &= 
  \begin{bmatrix}
    \vv & \mU_1{\mZ}^{(1)}\mP_1^{\T}  & \mU_1{\mZ}^{(2)}\mP_2^{\T} & \ldots & \mU_l{\mZ}^{(l)}\mP_l^{\T}
  \end{bmatrix}\\
  &=
  \begin{bmatrix}
    \vv & \mV_1 & \mV_2 & \ldots & \mV_l
  \end{bmatrix} \\
  &= \tilde{\mV}.
\end{align*}
Hence, we get 
\begin{align*}
  \begin{bmatrix}
    1 & \0 \\
    \0 & \mP
  \end{bmatrix}
  \begin{bmatrix}
    1 & \vt^{\T} \\
    \vt & \mT
  \end{bmatrix}
  \begin{bmatrix}
    1 & \0 \\
    \0 & \mP^{\T}
  \end{bmatrix} & = 
  \begin{bmatrix}
    1 & \0 \\
    \0 & \mP
  \end{bmatrix}
  \mR^{\T}\mR
  \begin{bmatrix}
    1 & \0 \\
    \0 & \mP^{\T}
  \end{bmatrix} \\
  & = \tilde{\mV}^{\T}\tilde{\mV}\\
  & =    \begin{bmatrix}
    1 & \vy^{\T} \\
    \vy & \mY
  \end{bmatrix},
\end{align*}
which is equivalent to $\mY=\mP\mT\mP^{\T}$ and $\vy=\mP\vt$. 
Lastly, note that 
\begin{align}
  \begin{bmatrix}
    1 & \vt[\beta_i]^{\T} \\
    \vt[\beta_i] & \mT[\beta_i]
\end{bmatrix}
&= 
\begin{bmatrix}
  \vv & \mU_i\mZ^{(i)}
\end{bmatrix}^{\T}
\begin{bmatrix}
  \vv & \mU_i\mZ^{(i)}
\end{bmatrix} \label{eq:submatrix_of_T}\\ 
&= \begin{bmatrix}
  \mU_i\vz^{(i)} & \mU_i\mZ^{(i)}
\end{bmatrix}^{\T}
\begin{bmatrix}
  \mU_i\vz^{(i)} & \mU_i\mZ^{(i)}
\end{bmatrix} \label{eq:substitute v_1}\\
&= (\tilde{\mZ}^{(i)})^{\T}\mU_i^{\T}\mU_i\tilde{\mZ}^{(i)} \nonumber\\ 
& = (\tilde{\mZ}^{(i)})^{\T}\tilde{\mZ}^{(i)} \label{eq:orthogonal of U}\\
& = \tilde{\mT}^{(i)} = 
\begin{bmatrix}
  1 & (\vt^{(i)})^{\T} \\
  \vt^{(i)} & \mT^{(i)}
\end{bmatrix},\nonumber
\end{align}
where we used \cref{eq:constructed t and T} in \cref{eq:submatrix_of_T}, $\vv=\mU_i\vz^{(i)}$ (cf. \cref{eq: equality between V and Z}) in \cref{eq:substitute v_1}, and 
$\mU_i^{\T}\mU_i=\mI_{d_i+1}$ in \cref{eq:orthogonal of U}. Hence, the remaining 
constraints in \cref{eq:original} are also satisfied because of the conditions on
$\tilde{\mT}^{(i)}$ in \cref{eq:decomposed}. 

The proof of \cref{thm:main equiv thm} is complete.  
\end{proof}

Two remarks are in order. Firstly, the variables in \cref{eq:original} are in 
a high-dimensional PSD cone $\semiS^{d+1}$, while those in \cref{eq:decomposed} are in the 
Cartesian product of smaller PSD cones 
$
\semiS^{k+1}\times \semiS^{d_1+1}\times \semiS^{d_2+1}\times \cdots \times \semiS^{d_l+1}
$.  
When $k,\;d_1,\; d_2,\ldots,d_l$ are much smaller than $d$, using \cref{eq:decomposed} instead of 
\cref{eq:original} can achieve dimension reduction without any  
additional cost. This can bring substantially higher computational efficiency for solving 
the corresponding SDP in practice (see \cref{sec:experiments}). 
Secondly, \cref{thm:main equiv thm} is very general and thus could be applicable to a potentially wide range 
of problems. It is worth highlighting that we require no assumptions on the sets $\mathcal{A}_i$ that constrain the submatrices, as well 
as the row and column partitions of the separable matrix $\mP$. 
This enables us to accommodate both the equivalence results 
\cref{eq:modelIII-ERSDP equiv} and \cref{eq:VA-BC equiv}, as we will 
show next. 
\subsubsection{Equivalence between \texorpdfstring{\cref{eq:Model III}}{(ESDR2-T)} and 
\texorpdfstring{\cref{eq:ERSDP}}{(ESDR-Y)}}
As we noted before, the transformation matrix $\hat{\mS}$ in \cref{eq:modelIII-ERSDP equiv} is separable with the row 
and column partitions given in \cref{eq:parititon in ModelIII}
and we have 
\begin{equation*}
  \hat{\mS}[\alpha_i,\beta_i]=
  \begin{bmatrix}
    \vs^{\T}_{R} \\
    \vs^{\T}_{I}
  \end{bmatrix}\in \R^{2\times M},\;i=1,2,\ldots,n.
\end{equation*}
Moreover, we can see that the feasible set of \cref{eq:Model III} 
is in the form of \cref{eq:original} with the set $\mathcal{A}_i$ given by 
\begin{align*} 
  \mathcal{A}_i &= 
  \biggl\{
    \begin{bmatrix}
      1 & \vt^{\T} \\
      \vt & \Diag(\vt)
    \end{bmatrix}:\;\vt\in\R^M,\;
    \sum_{j=1}^M t_j=1,\;t_j\geq 0,\;j=1,2,\ldots,M
    \biggr\} \\
  &= \biggl\{
    \sum_{j=1}^M t_j\mE_j:\;\sum_{j=1}^M t_j=1,\;t_j\geq 0,\;j=1,2,\ldots,M
    \biggr\} \\
    &= \mathrm{conv}\{\mE_1,\mE_2,\ldots,\mE_M\}, 
\end{align*}
where  
\begin{equation*}
  \mE_j = 
  \begin{bmatrix}
    1 \\
    \ve_j
  \end{bmatrix}
  \begin{bmatrix}
    1 \\
    \ve_j
  \end{bmatrix}^{\T},\;j=1,2,\ldots,M,
\end{equation*}
and $\ve_j\in\R^M$ is the $j$-th unit vector. 
Applying \cref{thm:main equiv thm} to 
\cref{eq:Model III} gives the following equivalent formulation: 
\begin{equation}\label{eq:equivalent of Model III} 
  \begin{aligned}
    & \begin{bmatrix}
      1 & \vy^{\T} \\
      \vy & \mY
    \end{bmatrix}\in \semiS^{2n+1} \\
    \mathrm{s.t.}\;\;\; & 
    \begin{bmatrix}
      1 & y_i & y_{n+i} \\
      y_i & Y_{i,i} & Y_{i,n+i} \\
      y_{n+i} & Y_{n+i,i} & Y_{n+i,n+i}
    \end{bmatrix}
    =\begin{bmatrix}
      1 & \0  \\
      \0  & \vs^{\T}_{R} \\
      \0  & \vs^{\T}_{I}
    \end{bmatrix}
    \begin{bmatrix}
      1 & (\vt^{(i)})^{\T} \\
      \vt^{(i)} & \mT^{(i)}
    \end{bmatrix}
    \begin{bmatrix}
      1 & \0  & \0  \\
      \0  & \vs_{R} & \vs_{I}
    \end{bmatrix}, \\
    & \begin{bmatrix}
      1 & (\vt^{(i)})^{\T} \\
      \vt^{(i)} & \mT^{(i)}
    \end{bmatrix}\in \semiS^{M+1},\\
    & \begin{bmatrix}
      1 & (\vt^{(i)})^{\T} \\
      \vt^{(i)} & \mT^{(i)}
    \end{bmatrix}\in \mathrm{conv}\{\mE_1,\mE_2,\ldots,\mE_M\},\;
    i=1,2,\ldots,n.
  \end{aligned}
\end{equation}
Since each matrix
$\mE_j$ is PSD, their convex hull is a subset of 
$\semiS^{M+1}$ and hence the PSD constraints in \cref{eq:equivalent of Model III} are 
redundant. Furthermore, 
note that 
\begin{align*}
  \begin{bmatrix}
    1 & \0  \\
    \0  & \vs^{\T}_{R} \\
    \0  & \vs^{\T}_{I}
  \end{bmatrix}\mE_j
  \begin{bmatrix}
    1 & \0  & \0  \\
    \0  & \vs_{R} & \vs_{I}
  \end{bmatrix}
  &=\begin{bmatrix}
    1 & \0  \\
    \0  & \vs^{\T}_{R} \\
    \0  & \vs^{\T}_{I}
  \end{bmatrix}
  \begin{bmatrix}
    1 \\
    \ve_j
  \end{bmatrix}
  \begin{bmatrix}
    1 \\
    \ve_j
  \end{bmatrix}^{\T}
  \begin{bmatrix}
    1 & \0  & \0  \\
    \0  & \vs_{R} & \vs_{I}
  \end{bmatrix}\\
  &=
  \begin{bmatrix}
    1 \\  s_{R,j} \\  s_{I,j}
  \end{bmatrix}
  \begin{bmatrix}
    1 & s_{R,j} & s_{I,j}
  \end{bmatrix},
\end{align*}
which is exactly the matrix $\mK_j$ defined in \cref{eq:def of K}. Therefore, 
we can conclude that \cref{eq:equivalent of Model III} is the same as 
\cref{eq:ERSDP}, and hence 
\cref{eq:Model III} and \cref{eq:ERSDP} are equivalent. 
\subsubsection{Equivalence between \texorpdfstring{\cref{eq:VA-SDR}}{(VA-SDR)} and 
\texorpdfstring{\cref{eq:BC-SDR}}{(BC-SDR)}}
Similarly, we observe that the transformation matrix $\mW$ in \cref{eq:VA-BC equiv} is separable with row and column 
partitions given in \cref{eq:parititon in VA}, and let 
\begin{equation}\label{eq:def of w}
  \vw^{\T} := \mW[\alpha_i,\beta_i]=
  \begin{bmatrix}
    1 & 2 & 4 & \ldots & 2^{q-1}
  \end{bmatrix}. 
\end{equation}
The feasible set in \cref{eq:VA-SDR} conforms to \cref{eq:original} with the 
set $\mathcal{A}_i$ given by 
\begin{equation*}
  \mathcal{A}_i = 
  \biggl\{
  \begin{bmatrix}
    1 & \vb^{\T} \\
    \vb & \mB
  \end{bmatrix}:\;
  \vb\in \R^q,\;\mB\in\R^{q\times q},\;
  \diag(\mB) = \1_q
  \biggr\}.
\end{equation*}
Hence, by applying \cref{thm:main equiv thm} to \cref{eq:VA-SDR}, we get 
the following equivalent formulation:  
\begin{equation}\label{eq:equivalent of VA-SDR} 
  \begin{aligned}
    & \begin{bmatrix}
      1 & \vx^{\T} \\
      \vx & \mX
    \end{bmatrix}\in \semiS^{n+1} \\
    \mathrm{s.t.}\;\;\; & 
    X_{i,i} = \vw^{\T} \mB^{(i)}\vw,\; x_i=\vw^{\T}\vb^{(i)},\\
    & \begin{bmatrix}
      1 & (\vb^{(i)})^{\T} \\
      \vb^{(i)} & \mB^{(i)}
    \end{bmatrix}\in \semiS^{q+1},\\
    &  
    \diag(\mB^{(i)}) = \1_q,\;
    i=1,2,\ldots,n.
  \end{aligned}
\end{equation}

Next we argue that all the constraints $x_i=\vw^{\T}\vb^{(i)}$ are redundant, i.e., 
the set in \cref{eq:equivalent of VA-SDR} is equivalent to 
\begin{equation}\label{eq:equivalent of VA-SDR 2}
  \biggl\{
  \begin{bmatrix}
    1 & \vx^{\T} \\
    \vx & \mX
  \end{bmatrix}\in\semiS^{n+1}:\;
  X_{i,i} = \vw^{\T}\mB^{(i)}\vw,\;
  \mB^{(i)}\in \semiS^q,\;
  \diag(\mB^{(i)}) = \1_q,\;
  i=1,2,\ldots,n\biggr\}.
\end{equation}
To show this, we need to prove that, for any $\vx,\mX,\mB^{(1)},\ldots,\mB^{(n)}$ 
satisfying the constraints in \cref{eq:equivalent of VA-SDR 2}, there must exist  
$\vb^{(i)}\in \R^q$ such that 
\begin{equation}\label{eq:conditions on b}
  x_i=\vw^{\T}\vb^{(i)}\;\text{and}\;\begin{bmatrix}
    1 & (\vb^{(i)})^{\T} \\
    \vb^{(i)} & \mB^{(i)}
  \end{bmatrix}\succeq 0,\;i=1,2,\ldots,n.
\end{equation}
Fix $i\in\{1,2,\ldots,n\}$. Note that the PSD constraints in \cref{eq:equivalent of VA-SDR 2} implies 
\begin{equation}\label{eq: inequality on x_i and X_ii}
  \begin{bmatrix}
    1 & x_i \\
    x_i & X_{i,i}
  \end{bmatrix}\succeq 0 \Leftrightarrow 
  x_i^2\leq X_{i,i}.
\end{equation}
When $X_{i,i}=0$, we must have $x_i=0$ and we can achieve \cref{eq:conditions on b} by 
simply letting $\vb^{(i)}=\0$. Otherwise, we have $X_{i,i}>0$ and hence we can let 
\begin{equation}\label{eq:equality of b}
  \vb^{(i)} = \frac{x_i}{X_{i,i}}\mB^{(i)}\vw. 
\end{equation}
Since $X_{i,i}=\vw^{\T}\mB^{(i)}\vw$, we can see that $\vw^{\T}\vb^{(i)}=(\vw^{\T}\mB^{(i)}\vw)x_i/X_{i,i}=x_i$. 

To verify the PSD constraint in \cref{eq:conditions on b}, it suffices to show that 
$
  \mB^{(i)}\succeq \vb^{(i)}(\vb^{(i)})^{\T}
$. 
Note that 
\begin{equation*}
  \begin{bmatrix}
    X_{i,i} & \vw^{\T}\mB^{(i)} \\
    \mB^{(i)}\vw  & \mB^{(i)}
  \end{bmatrix}=
  \begin{bmatrix}
    \vw^{\T}\mB^{(i)}\vw & \vw^{\T}\mB^{(i)} \\
    \mB^{(i)}\vw  & \mB^{(i)}
  \end{bmatrix}
  =\begin{bmatrix}
    \vw^{\T} \\
    \mI_{q}
  \end{bmatrix}\mB^{(i)}
  \begin{bmatrix}
    \vw & \mI_{q} 
  \end{bmatrix}\succeq 0,
\end{equation*}
which implies the Schur complement is also PSD, i.e.,  
\begin{equation*}
  \mB^{(i)}-\frac{1}{X_{i,i}}\mB^{(i)}\vw\vw^{\T}\mB^{(i)}\succeq 0. 
\end{equation*}
This, together with \cref{eq: inequality on x_i and X_ii,eq:equality of b}, shows
\begin{equation*}
  \vb^{(i)}(\vb^{(i)})^{\T} =\frac{x_i^2}{X_{i,i}^2}\mB^{(i)}\vw\vw^{\T}\mB^{(i)}
  \preceq \frac{X_{i,i}}{X_{i,i}^2}\mB^{(i)}\vw\vw^{\T}\mB^{(i)} 
  \preceq \mB^{(i)}.
\end{equation*} 
Hence both conditions in \cref{eq:conditions on b} are satisfied. 

Finally, to show that \cref{eq:equivalent of VA-SDR 2} is the same as 
\cref{eq:BC-SDR}, we need the following lemma. 
\begin{lemma}\label{lem:bound on Xii}
  Let $\vw\in\R^q$ be the vector defined in \cref{eq:def of w}. It holds that  
  \begin{equation}\label{eq:set_equiv}
    \bigl\{x:\;\exists\;\mB\in\semiS^{q}\;\mathrm{s.t.}\; x=\vw^{\T}\mB\vw,\;    \diag(\mB) = \1_q
    \bigr\}
    =\{1\leq x\leq (2^q-1)^2\}.
  \end{equation}
\end{lemma}
\begin{proof}
  See \cref{appen:proof of interval}. 
\end{proof}

Putting all pieces together, we have proved that \cref{eq:VA-SDR} is equivalent to 
\cref{eq:BC-SDR} by showing the correspondence \cref{eq:VA-BC equiv}. 


\section{Numerical results}\label{sec:experiments}
In this section, we present some numerical results. 
Following standard assumptions in the wireless communication literature (see, e.g.,~\cite[Chapter 7]{tse2005fundamentals}), we 
assume that all entries of the channel matrix $\mH$ are 
independent and identically distributed (i.i.d.) 
following a complex circular Gaussian distribution with zero mean and 
unit variance, and all entries of the additive noise $\vv$ are 
i.i.d. following a complex circular Gaussian distribution with zero mean and 
variance $\sigma^2$. Further, we choose the transmitted symbols $x^*_1,x^*_2,\ldots,x^*_n$ 
from 
the symbol set $\sS_M$ in \cref{eq:MPSK symbol set} independently and uniformly.
We define the SNR as the received SNR per symbol: 
\begin{equation*}
  \mathrm{SNR}:= \frac{\E[\|\mH\vx^*\|_2^2]}{n\E[\|\vv\|_2^2]} = \frac{mn}{n\cdot m\sigma^2}=\frac{1}{\sigma^2}. 
\end{equation*}

We first consider a MIMO system where $(m,n)=(16,10)$ and $M=8$. 
To evaluate the empirical probabilities of SDRs not being tight, 
we compute the optimal solutions 
of \cref{eq:CSDP2}, \cref{eq:ERSDP}, and \cref{eq:Model II} by the general-purpose SDP solver SeDuMi~\cite{sedumi} with the desired accuracy set to $10^{-6}$. 
The SDR is decided to be tight if the output $\hat{\vx}$ returned by the SDP solver\footnote{The output $\hat{\vx}$ is directly given by the optimal solution in \cref{eq:CSDP2}, 
while it is obtained from the relation \cref{eq:def of hat Q and hat c} between $\vx$ and $\vy$ in \cref{eq:ERSDP} and 
the relation \cref{eq:express discrete by 01} between $\vx$ and $\vt$ in \cref{eq:Model II}.} satisfies $\|\hat{\vx}-\vx^*\|_\infty\leq 10^{-4} $.
We also evaluate the empirical probabilities of conditions \cref{eq:CSDP2_suff,eq:CSDP2_necc_and_suff_intro,eq:ERSDP_necc_intro} 
not being satisfied. 
We run the simulations at 8 SNR values in total ranging from 3 dB to 24 dB.  
For each SNR value, 10,000 random instances are generated 
and the averaged results are plotted in \cref{fig:tightness probability versus SNR}. 
  
\begin{figure}[htbp]
  \centering
  \includegraphics[width=0.9\linewidth]{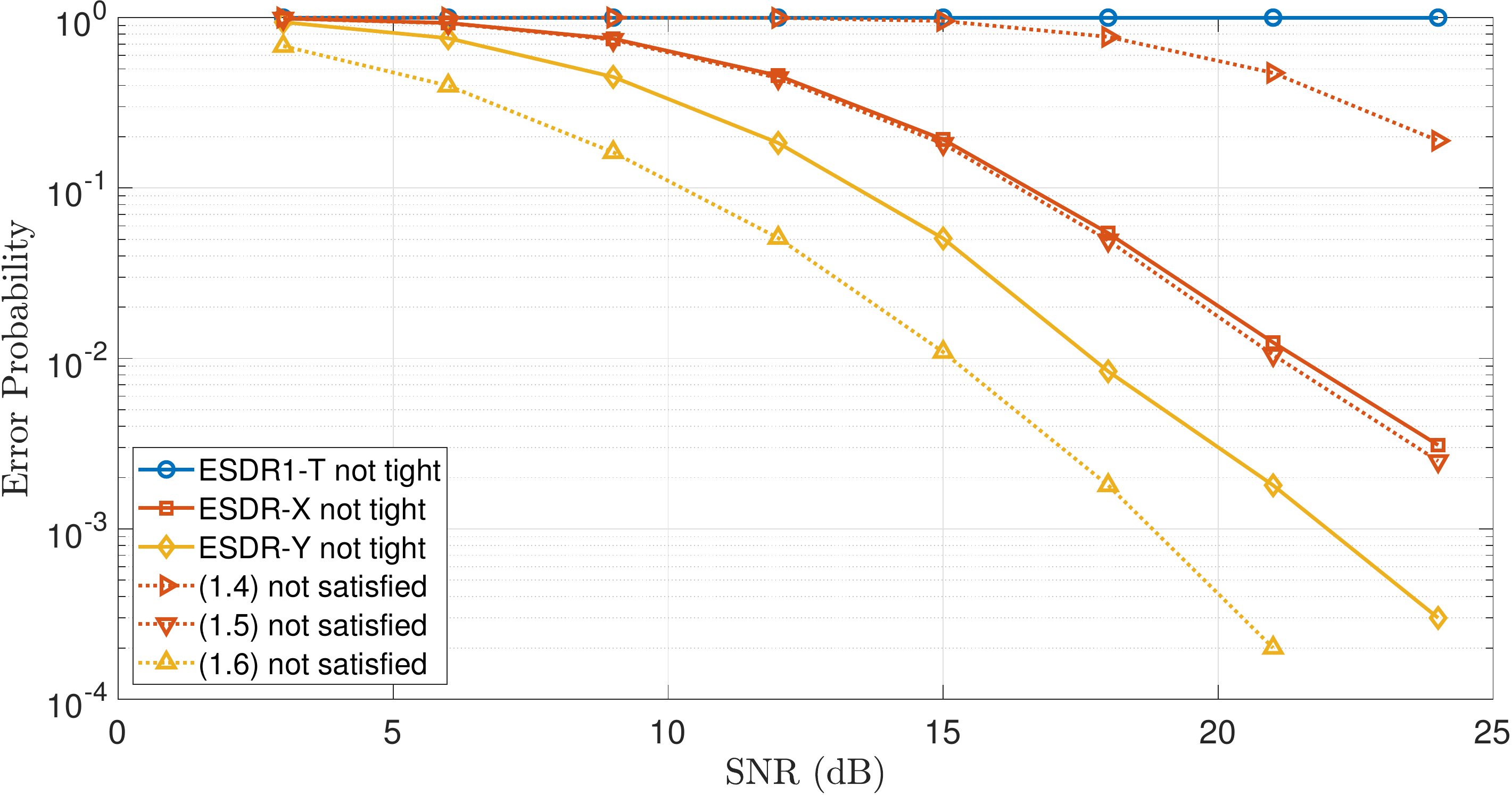}
  \caption{Error probabilities versus the SNR in a $16\times 10$ MIMO system with 8-PSK. }
  \label{fig:tightness probability versus SNR}
\end{figure}

We can see from \cref{fig:tightness probability versus SNR} that our results 
\cref{eq:CSDP2_necc_and_suff_intro,eq:ERSDP_necc_intro} provide better 
characterizations than the previous tightness condition \cref{eq:CSDP2_suff} in~\cite{liu2019equivalence}. 
The empirical probability of \cref{eq:CSDP2} not being tight 
matches perfectly with our analysis given by 
the necessary and sufficient condition \cref{eq:CSDP2_necc_and_suff_intro}. 
The probability of \cref{eq:ERSDP_necc_intro} not being satisfied is also 
a good approximation to the probability of \cref{eq:ERSDP} not being tight, 
underestimating the latter roughly by a factor of 9.
Moreover, the numerical results also validate our analysis that 
\cref{eq:Model II} is not tight with high probability. 
In fact, 
\cref{eq:Model II} fails to recover the vector of transmitted symbols 
in all 80,000 instances. 


Next, we compare the optimal values as well as the CPU time of solving
\cref{eq:ERSDP} and \cref{eq:Model III}. \cref{tab:relative diff} shows the relative 
difference between the optimal values of \cref{eq:ERSDP} (denoted as $\mathrm{opt}_{\text{ESDR-}\mY}$) and \cref{eq:Model III} (denoted as $\mathrm{opt}_{\text{ESDR2-}\mT}$) averaged over 300 simulations, which is defined as 
$|\mathrm{opt}_{\text{ESDR-}\mY}-\mathrm{opt}_{\text{ESDR2-}\mT}|/|\mathrm{opt}_{\text{ESDR2-}\mT}|$. 
We can see from \cref{tab:relative diff} that the difference is consistently in the order 1e$-$7 in various settings, which verifies the equivalence between 
\cref{eq:ERSDP} and \cref{eq:Model III}. 
In \cref{fig:time complexity}, we plot the average CPU time consumed by solving \cref{eq:ERSDP} and \cref{eq:Model III} in an 
8-PSK system with increasing problem size $n$. For fair comparison, both SDRs are implemented and solved by SeDuMi and we repeat 
the simulations for 300 times. With the same error performance, 
we can see that \cref{eq:ERSDP} indeed solves the MIMO detection problem \cref{eq:ML} more efficiently and saves roughly 90\% of the computational time in 
our experiment. 
  \begin{table}[tbhp]\label{tab:relative diff}
  {\footnotesize
  \caption{Average relative difference between \cref{eq:ERSDP} and \cref{eq:Model III} in optimal objective values.}
  \begin{center}    
  \begin{tabular}{|c|c|c|c|c|}
    \hline
    \multirow{2}{*}{SNR}       & \multicolumn{4}{c|}{Relative diff. in optimal objective values.}                           \\ \cline{2-5} 
                               & $(m,n)= (4,4)$               & $(m,n)=(6,4)$                & $(m,n)=(10,10)$                & $(m,n)=(15,10)  $              \\ \hline
    5dB                       &      4.62e$-$7                  & 5.10e$-$7                      & 6.26e$-$7                      &   7.73e$-$7                    \\ \hline
    10dB                       &     5.67e$-$7                  & 4.06e$-$7                      & 6.50e$-$7                      &   7.16e$-$7                    \\ \hline
    {15dB} & 5.94e$-$7 & 3.83e$-$7 & 7.80e$-$7 & 5.58e$-$7 \\ \hline
    \end{tabular}
  \end{center}
  }
  \end{table}
\begin{figure}[htbp]
  \centering
  \includegraphics[width=0.9\linewidth]{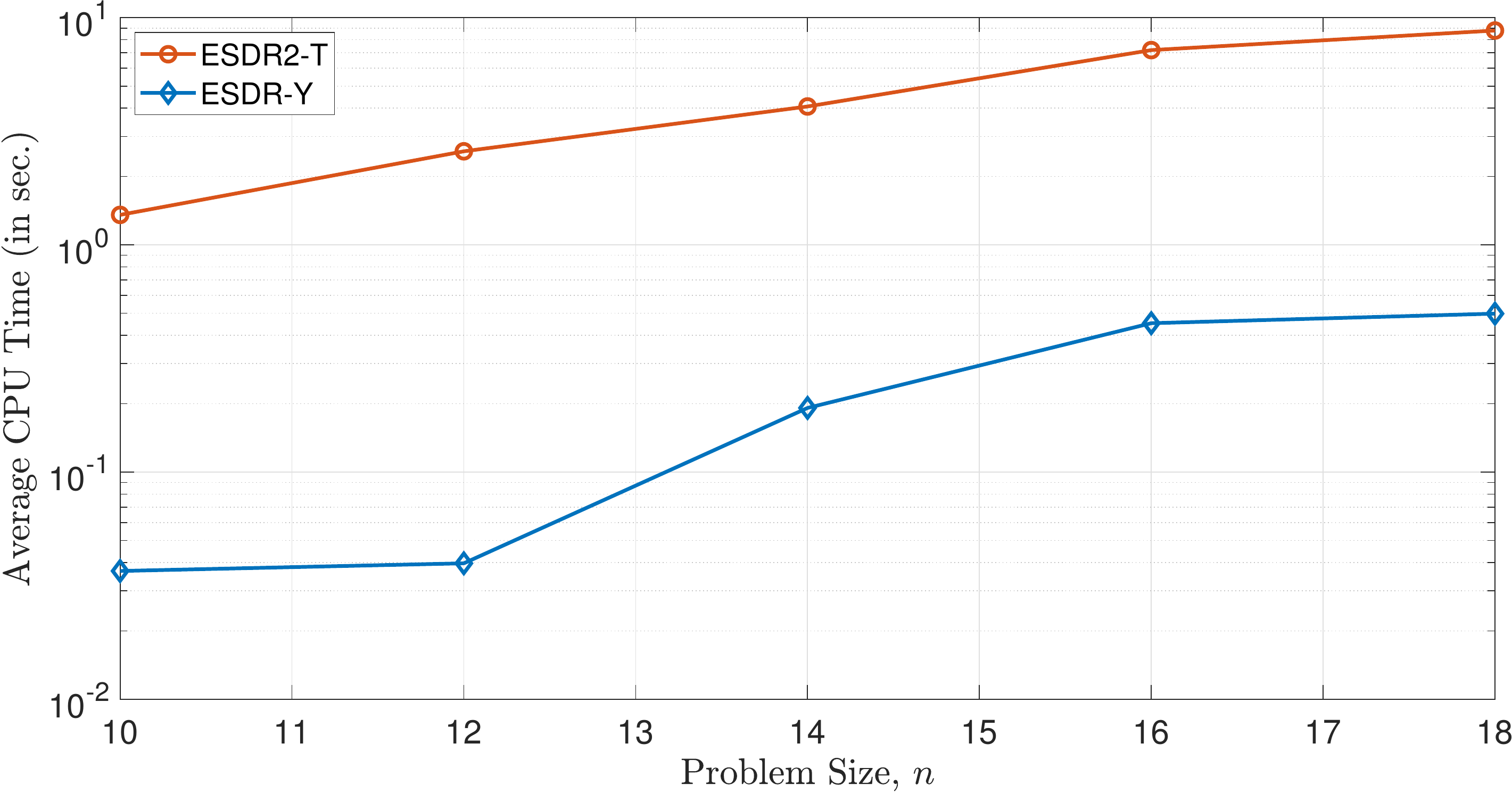}
  \caption{Average CPU time of solving \cref{eq:ERSDP} and \cref{eq:Model III} when $M=8$.}
  \label{fig:time complexity}
\end{figure}

\section{Conclusions}
\label{sec:conclusions}
In this paper, we studied the tightness and 
equivalence of various existing SDR models for the MIMO detection 
problem \cref{eq:ML}. 
For the two SDRs \cref{eq:CSDP2} and \cref{eq:ERSDP} proposed in~\cite{tightness_enhanced}, 
we improved their sufficient tightness condition and  
showed that the former is tight if and only if \cref{eq:CSDP2_necc_and_suff_intro} holds 
while the latter is tight only if \cref{eq:ERSDP_necc_intro} holds. 
On the other hand, for the SDR \cref{eq:Model II} proposed in~\cite{near_ml_decoding}, 
we proved that its tightness probability decays to zero exponentially 
fast with an increasing problem size under some mild assumptions. 
Together with known results, our analysis provides a more complete understanding
of the tightness conditions for existing SDRs. 
Moreover, we proposed a general theorem that 
unifies previous results on the equivalence of SDRs~\cite{QAM_equivalence,liu2019equivalence}. 
For a subset of PSD matrices with a special ``separable'' structure, 
we showed its equivalence to another subset of PSD matrices in 
a potentially much smaller dimension.  
Our numerical results demonstrated 
that we 
could significantly improve the computational efficiency by using such equivalence. 

Due to its generality, 
we believe that our equivalence theorem can be applied to SDPs in other domains beyond MIMO detection 
and we would like to put this as a future work. 
Additionally, we noticed that  
the SDRs for problem \cref{eq:ML} combined with some simple rounding 
procedure can detect the transmitted symbols successfully even when  
the optimal solution has rank more than one. Similar observations have also been made in~\cite{diversity_SDR}. 
It would be interesting to extend our analysis to take the postprocessing procedure into account.  
\appendix
\section{Simplification of \texorpdfstring{\cref{eq:ERSDP equality,eq:ERSDP inequality}}{(3.5) and (3.6)}}
\label{appen:simplification}  
Fix $i\in\{1,2,\ldots,n\}$. From \cref{eq:ERSDP equality}, we have 
\begin{equation*}
  (\hat{\mH}^{\T}\hat{\vv})_i = g_i+\lambda_iy_i^*+\mu_iy_{n+i}^*\;
  \text{and}\;
  (\hat{\mH}^{\T}\hat{\vv})_{n+i} = g_{n+i}+\mu_iy_i^*+\lambda_{n+i}y_{n+i}^*,
\end{equation*}
which can be written in a matrix form:
\begin{equation}\label{eq:ERSDP equality matrix form}
  \begin{bmatrix}
    (\hat{\mH}^{\T}\hat{\vv})_i \\
    (\hat{\mH}^{\T}\hat{\vv})_{n+i}
  \end{bmatrix}
  =
  \begin{bmatrix}
    g_i \\
    g_{n+i}
  \end{bmatrix}
  +
  \begin{bmatrix}
    \lambda_i & \mu_i \\
    \mu_i & \lambda_{n+i}
  \end{bmatrix}
  \begin{bmatrix}
    y_i^* \\
    y_{n+i}^*
  \end{bmatrix}.
\end{equation}
Recall the definitions of $\mGamma_i$ in \cref{eq:def_of_Lambda_M_Gamma} and $\mK_j$ in \cref{eq:def of K}. 
Then
\begin{equation}\label{eq:gamma dot P}
    \langle \mGamma_i,\mK_j\rangle = 
    2  
    \begin{bmatrix}
    \cos(\theta_j) \\ \sin(\theta_j)
    \end{bmatrix}^{\T}
    \begin{bmatrix}
    g_i \\ g_{n+i}
    \end{bmatrix}+
    \begin{bmatrix}
    \cos(\theta_j) \\ \sin(\theta_j)
    \end{bmatrix}^{\T}
    \begin{bmatrix}
    \lambda_i & \mu_i \\
    \mu_i & \lambda_{n+i}  
    \end{bmatrix}
    \begin{bmatrix}
      \cos(\theta_j) \\ \sin(\theta_j)
    \end{bmatrix}.
\end{equation} 
Using \cref{eq:ERSDP equality matrix form}, we have
\begin{equation}\label{eq:substitute g}
  \begin{aligned}
    \begin{bmatrix}
      \cos(\theta_j) \\ \sin(\theta_j)
      \end{bmatrix}^{\T}
      \begin{bmatrix}
      g_i \\ g_{n+i}
      \end{bmatrix}
      &= 
      \begin{bmatrix}
        \cos(\theta_j) \\ \sin(\theta_j)
      \end{bmatrix}^{\T}
      \begin{bmatrix}
        (\hat{\mH}^{\T}\hat{\vv})_i \\
        (\hat{\mH}^{\T}\hat{\vv})_{n+i}
      \end{bmatrix}
      -
      \begin{bmatrix}
        \cos(\theta_j) \\ \sin(\theta_j)
      \end{bmatrix}^{\T}
      \begin{bmatrix}
        \lambda_i & \mu_i \\
        \mu_i & \lambda_{n+i}
      \end{bmatrix}
      \begin{bmatrix}
        y_i^* \\
        y_{n+i}^*
      \end{bmatrix} \\
      &= 
      \begin{bmatrix}
        \cos(\Delta\theta_j) \\ \sin(\Delta\theta_j)
      \end{bmatrix}^{\T}
      \begin{bmatrix}
        \hat{z}_i \\
        \hat{z}_{n+i}
      \end{bmatrix}
      -
      \begin{bmatrix}
        \cos(\theta_j) \\ \sin(\theta_j)
      \end{bmatrix}^{\T}
      \begin{bmatrix}
        \lambda_i & \mu_i \\
        \mu_i & \lambda_{n+i}
      \end{bmatrix}
      \begin{bmatrix}
        y_i^* \\
        y_{n+i}^*
      \end{bmatrix},
  \end{aligned}
\end{equation} 
where $\Delta\theta_j=\theta_j-\theta_{u_i}$, $\hat{z}_i=\Real(z_i)$, 
and $\hat{z}_{n+i}=\Imag(z_i)$ (cf. \cref{eq:def of z}).
Combining \cref{eq:gamma dot P} with \cref{eq:substitute g}, we get 
\begin{align*}
  \langle \mGamma_i,\mK_j\rangle &= 
  2\begin{bmatrix}
    \cos(\Delta\theta_j) \\ \sin(\Delta\theta_j)
  \end{bmatrix}^{\T}
  \begin{bmatrix}
    \hat{z}_i \\
    \hat{z}_{n+i}
  \end{bmatrix}
  -
  2\begin{bmatrix}
    \cos(\theta_j) \\ \sin(\theta_j)
  \end{bmatrix}^{\T}
  \begin{bmatrix}
    \lambda_i & \mu_i \\
    \mu_i & \lambda_{n+i}
  \end{bmatrix}
  \begin{bmatrix}
    y_i^* \\
    y_{n+i}^*
  \end{bmatrix}+\\
  &\phantom{=}\; \begin{bmatrix}
    \cos(\theta_j) \\ \sin(\theta_j)
    \end{bmatrix}^{\T}
    \begin{bmatrix}
    \lambda_i & \mu_i \\
    \mu_i & \lambda_{n+i}  
    \end{bmatrix}
    \begin{bmatrix}
      \cos(\theta_j) \\ \sin(\theta_j)
    \end{bmatrix}.
\end{align*}
In particular, when $j=u_i$, the above becomes 
\begin{equation*}
  \langle \mGamma_i,\mK_{u_i}\rangle =
  2\begin{bmatrix}
    1 \\ 0
  \end{bmatrix}^{\T}
  \begin{bmatrix}
    \hat{z}_i \\
    \hat{z}_{n+i}
  \end{bmatrix}
  -
  \begin{bmatrix}
    y_i^* \\
    y_{n+i}^*
  \end{bmatrix}^{\T}
    \begin{bmatrix}
    \lambda_i & \mu_i \\
    \mu_i & \lambda_{n+i}  
    \end{bmatrix}
    \begin{bmatrix}
      y_i^* \\
      y_{n+i}^*
    \end{bmatrix}.
\end{equation*}
Hence, when $j\neq u_i$, \cref{eq:ERSDP inequality} is equivalent to 
\begin{align*}
  2\begin{bmatrix}
    1-\cos(\Delta\theta_j) \\ -\sin(\Delta\theta_j)
  \end{bmatrix}^{\T}
  \begin{bmatrix}
    \hat{z}_i \\
    \hat{z}_{n+i}
  \end{bmatrix}
  &\geq 
  \begin{bmatrix}
    y_i^*-\cos(\theta_j) \\
    y_{n+i}^*-\sin(\theta_j)
  \end{bmatrix}^{\T}
    \begin{bmatrix}
    \lambda_i & \mu_i \\
    \mu_i & \lambda_{n+i}  
    \end{bmatrix}
    \begin{bmatrix}
      y_i^*-\cos(\theta_j) \\
      y_{n+i}^*-\sin(\theta_j)
    \end{bmatrix} \\
  \Leftrightarrow 
  \begin{bmatrix}
    1 \\ -\cot\left(\frac{\Delta\theta_j}{2}\right)
  \end{bmatrix}^{\T}
  \begin{bmatrix}
    \hat{z}_i \\
    \hat{z}_{n+i}
  \end{bmatrix}
  &\geq 
  \begin{bmatrix}
    \sin(\theta_{u_i}+\frac{\Delta\theta_j}{2}) \\
    -\cos(\theta_{u_i}+\frac{\Delta\theta_j}{2})
  \end{bmatrix}^{\T}
    \begin{bmatrix}
    \lambda_i & \mu_i \\
    \mu_i & \lambda_{n+i}  
    \end{bmatrix}
    \begin{bmatrix}
      \sin(\theta_{u_i}+\frac{\Delta\theta_j}{2}) \\
      -\cos(\theta_{u_i}+\frac{\Delta\theta_j}{2})
    \end{bmatrix},
\end{align*} 
which is exactly the same as \cref{eq:simplified inequality}.

\section{Proof of \texorpdfstring{\cref{lem:bound on Xii}}{Lemma 4.4}}\label{appen:proof of interval}
To simplify the notations, we use $\mathcal{A}$ and $\mathcal{B}$ to denote the left-hand side and 
the right-hand side in \cref{eq:set_equiv}, respectively. 

We first prove that $\mathcal{A}\supset \mathcal{B}$. Let 
$\mathcal{C}= \bigl\{\mB\in\semiS^{q}:\;\diag(\mB) = \1_q\bigr\}$,
and we can view $\mathcal{A}$ as the image of the convex set 
$\mathcal{C}$ under the affine mapping 
$\mB\mapsto \langle \mB, \vw\vw^{\T}\rangle$. Therefore, the set $\mathcal{A}$ is also convex. 
Moreover, note that both the rank-one matrices $\1_q \1_q^{\T}$ and 
$[\begin{smallmatrix}
  -\1_{q-1} \\ 1
\end{smallmatrix}][\begin{smallmatrix}
  -\1_{q-1} \\ 1
\end{smallmatrix}]^{\T}$
belong to $\mathcal{C}$. Direct computations show that 
\begin{align*}
  \vw^{\T} \1_q \1_q^{\T}\vw &= \biggl(\sum_{i=1}^q 2^{i-1}\biggr)^2 = (2^q-1)^2,\\
  \vw^{\T} 
  \begin{bmatrix}
    -\1_{q-1}  \\ 1
  \end{bmatrix}
  \begin{bmatrix}
    -\1_{q-1}  \\ 1
  \end{bmatrix}^{\T}
  \vw &= \biggl(2^{q-1}-\sum_{i=1}^{q-1} 2^{i-1}\biggr)^2 = 1, 
\end{align*}
and hence both $1$ and $(2^q-1)^2$ belong to $\mathcal{A}$. 
Finally, the convexity of $\mathcal{A}$ implies $\mathcal{B}\subset \mathcal{A}$. 

Now we prove the other direction, i.e., $\mathcal{A}\subset \mathcal{B}$. This is equivalent to 
showing
\begin{equation*}
  1\leq \vw^{\T}\mB\vw \leq (2^q-1)^2,\;\forall\; \mB\in \mathcal{C}.
\end{equation*}
For the upper bound, we first note that $\mB\in \semiS^q$ implies 
\begin{equation}\label{eq:bound for each entry of B}
  |B_{i,j}|\leq \sqrt{B_{i,i}B_{j,j}}=1,\;\;\;i\neq j,\;1\leq i,j\leq q.
\end{equation}
Since every entry of the matrix $\vw\vw^{\T}$ is positive, we have 
\begin{equation*}
  \vw^{\T}\mB\vw = \langle \mB, \vw\vw^{\T}\rangle \leq
  \langle \1\1^{\T}, \vw\vw^{\T}\rangle = (2^q-1)^2
\end{equation*}
for any $\mB\in \mathcal{C}$, and hence the upper bound holds. 

For the lower bound, 
it clearly holds when $q=1$. 
When $q>1$,
for any matrix ${\mB}\in \mathcal{C}$ we 
partition it as 
\begin{equation*}
  {\mB} = 
  \begin{bmatrix}
    \mB' & \vb' \\
    (\vb')^{\T} & 1
  \end{bmatrix},
\end{equation*}
where $\mB'\in \R^{(q-1)\times (q-1)}$ and $\vb'\in \R^{q-1}$. 
Note that we have $|b'_i|\leq 1$ for 
$1\leq i\leq q-1$ (cf. \cref{eq:bound for each entry of B}), and 
$\mB\in \semiS^q$ implies $\mB'\succeq \vb'(\vb')^{\T}$. 
Further, we let 
$ {\vw}' = \begin{bmatrix}
    1 & 2 & 4 & \ldots & 2^{q-2}
  \end{bmatrix}^{\T} 
  \in \R^{q-1}$
such that $\vw=[(\vw')^{\T}\quad 2^{q-1}]^{\T}$ (cf. \cref{eq:def of w}). 
We have 
\begin{align*}
  {\vw}^{\T}{\mB}{\vw} &= 
  (\vw')^{\T}\mB\vw'+2^{q}(\vw')^{\T}\vb'+(2^{q-1})^2\\
  &\geq  
  (\vw')^{\T}\vb'(\vb')^{\T}\vw'+2^{q}(\vw')^{\T}\vb'+(2^{q-1})^2\\
   &= \bigl((\vw')^{\T}\vb'+2^{q-1}\bigr)^2.
\end{align*}
Since
\[(\vw')^{\T}\vb' = \sum_{i=1}^{q-1} 2^{i-1}b'_i \geq -\sum_{i=1}^{q-1} 2^{i-1} = -2^{q-1}+1,\] 
we immediately get 
${\vw}^{\T}{\mB}{\vw}\geq ( -2^{q-1}+1+2^{q-1})^2=1$,
and hence the lower bound also holds. 

The proof is now complete. 
\section*{Acknowledgments}
The authors would like to thank Professors Zi Xu and Cheng Lu for their useful discussions on an earlier version of this paper.

\bibliographystyle{siamplain}

\end{document}